\documentclass[12pt,reqno]{amsart}

\newcommand\version{February 26, 2009}

\usepackage{amsmath,amsfonts,amsthm,amssymb,amsxtra}

\newtheorem{theorem}{Theorem}[section]
\newtheorem{proposition}[theorem]{Proposition}
\newtheorem{lemma}[theorem]{Lemma}

\theoremstyle{definition}

\newtheorem{example}[theorem]{Example}

\theoremstyle{remark}

\newtheorem{remark}[theorem]{Remark}

\numberwithin{equation}{section}

\newcommand{\cl}{\mathrm{cl}}

\renewcommand{\epsilon}{\varepsilon}
\renewcommand{\hom}{\mathrm{hom}}

\newcommand{\loc}{{\rm loc}}

\newcommand{\N}{\mathbb{N}}

\newcommand{\R}{\mathbb{R}}

\newcommand{\Sph}{\mathbb{S}}

\renewcommand{\phi}{\varphi}

\newcommand{\Z}{\mathbb{Z}}

\DeclareMathOperator{\curl}{curl}
\DeclareMathOperator{\dist}{dist}

\DeclareMathOperator{\dom}{dom}
\DeclareMathOperator{\im}{Im}
\DeclareMathOperator{\ran}{ran}
\DeclareMathOperator{\re}{Re}
\DeclareMathOperator{\spec}{spec}

\DeclareMathOperator{\sgn}{sgn}
\DeclareMathOperator{\tr}{tr}

\begin{document}

\title[Eigenvalue estimates --- \version]{Remarks on eigenvalue estimates \\and semigroup domination}

\author{Rupert L. Frank}
\address{Rupert L. Frank, Department of Mathematics,
Princeton University, Washington Road, Princeton, NJ 08544, USA}
\email{rlfrank@math.princeton.edu}

\thanks{\copyright\, 2009 by the author. This paper may be reproduced, in its entirety, for non-commercial purposes.\\
Support through DFG grant FR 2664/1-1 and U.S. National Science Foundation grant PHY 06 52854 is gratefully acknowledged.}

\begin{abstract}
We present an overview over recent results concerning semi-classical spectral estimates for magnetic Schr\"odinger operators. We discuss how the constants in magnetic and non-magnetic eigenvalue bounds are related and we prove, in an abstract setting, that any non-magnetic Lieb-Thirring-type inequality implies a magnetic Lieb-Thirring-type inequality with possibly a larger constant.
\end{abstract}

\maketitle

\section{Introduction}

In this paper we review and extend some recent results concerning spectral estimates for magnetic Schr\"odinger operators. Let $d\geq2$, $\mathbf{A}$ a vector potential on $\R^d$ corresponding to the magnetic field $\curl \mathbf{A}$ and $V$ a real-valued and (in some averaged sense) decaying function on $\R^d$. Under rather general assumptions on $\mathbf{A}$ and $V$ one can define the self-adjoint Schr\"odinger operator $(\mathbf{D}-\mathbf{A})^2+V$ in $L_2(\R^d)$ through the closure of the quadratic form $\int_{\R^d} \left(|(\mathbf{D}-\mathbf{A})u|^2 +V|u|^2\right)\,dx$ on $C_0^\infty(\R^2)$. Here $\mathbf{D}=-i\nabla$. Our starting point is the well-known bound
\begin{equation}\label{eq:diamaglowest}
\inf\spec\left( (\mathbf{D}-\mathbf{A})^2+V \right) \geq \inf\spec \left( -\Delta+V \right) \,,
\end{equation}
which follows from the diamagnetic inequality
\begin{equation}
 \label{eq:diamagintro}
|\exp(-t \left((\mathbf{D}-\mathbf{A})^2+V\right) ) f| \leq \exp(-t \left(-\Delta+V\right) ) |f|
\quad \text{a.e.}
\end{equation}
for all $f\in L_2(\R^d)$; see \cite{Si79} for a proof under mild conditions on $\mathbf{A}$ and $V$.

While \eqref{eq:diamaglowest} concerns only the lowest eigenvalue, in this paper we are interested in the number, the sum or, more generally, moments of the negative eigenvalues of $(\mathbf{D}-\mathbf{A})^2+V$. That is, we will consider $\tr((\mathbf{D}-\mathbf{A})^2+V)_-^\gamma=\sum_{j} |\lambda_j((\mathbf{D}-\mathbf{A})^2+V)|^\gamma$, where $\lambda_j(H)$ is the $j$-th negative eigenvalue (taking multiplicities into account) of $H$ and $\gamma\geq 0$ is a parameter. For $\gamma=0$ this sum represents the number of negative eigenvalues. In particular, we are interested in bounds of the form
\begin{equation}\label{eq:ltintro}
\tr\left((\mathbf{D}-\mathbf{A})^2+V\right)_-^\gamma \leq L_{\gamma,d} \int_{\R^d} V_-^{\gamma+d/2} \,dx
\end{equation}
with a constant $L_{\gamma,d}$ independent of $V$ and $\mathbf{A}$. For $\mathbf{A}\equiv0$ and $\gamma>0$ if $d=2$ and $\gamma\geq 0$ if $d\geq 3$ these estimates are due to Lieb-Thirring \cite{LiTh}, Cwikel \cite{Cw}, Lieb \cite{Li2} and Rozenblum \cite{Ro2}; see \cite{LaWe2,Hu} for further references, applications and the problem of sharp constants. Using \eqref{eq:diamagintro} the proofs in \cite{LiTh} and \cite{Li2} can be extended to non-trivial $\mathbf{A}$.

It is a remarkable fact that any known proof of \eqref{eq:ltintro} which allows for the inclusion of a magnetic field, yields the same value for the constant $L_{\gamma,d}$ in the magnetic as in the non-magnetic case. Moreover, for $\gamma\geq 3/2$, when the sharp constant in \eqref{eq:ltintro} is known \cite{LaWe1}, the sharp constant for the magnetic inequality coincides with the sharp constant for the non-magnetic inequality. We also note that the semi-classical approximation to the left side of \eqref{eq:ltintro} is given by the phase space integral
$$
\iint_{\R^d\times\R^d} \left(|\xi-\mathbf{A}(x)|^2 +V(x)\right)_-^\gamma \,\frac{dx\,d\xi}{(2\pi)^d} \,,
$$
which is independent of $\mathbf{A}$! These observations lead to the question, whether the validity of \eqref{eq:ltintro} for $\mathbf{A}\equiv 0$ immediately implies its validity, with the same constant, for non-trivial $\mathbf{A}$.

This problem would be trivial if one had an analog of \eqref{eq:diamaglowest} for moments of eigenvalues. This is \emph{wrong}, however! Avron, Herbst and Simon \cite{AvHeSi} and Lieb \cite{Li3} (in a discrete model) have shown that a conjectured diamagnetic inequality for the number and the sum of eigenvalues fails.

In Theorem \ref{diamag} below we will prove a positive result: If \eqref{eq:ltintro} is valid for $\mathbf{A}\equiv 0$ with constant $L_{\gamma,d}$, then \eqref{eq:ltintro} is valid for \emph{any} $\mathbf{A}$ with a constant $\tilde L_{\gamma,d}=R_{\gamma,d} L_{\gamma,d}$, where $R_{\gamma,d}$ is an explicit constant depending only on $\gamma$ and $d$. For $\gamma=0$ this is a theorem of Rozenblum \cite{Ro}, while a simpler result for $\gamma>0$ has recently appeared in \cite{hltsimple} (see also \cite{stability} for a result for $\gamma=1$ in the case of operators with discrete spectrum). As in \cite{Ro}, we will prove a much more general result which is valid for an arbitrary pair of operators in $L_2$-spaces related by a diamagnetic inequality of the form \eqref{eq:diamagintro}.

Because of its generality this result can be applied in settings where the magnetic versions of the inequalities were previously not known; see Examples \ref{ex:ltlog}, \ref{ex:ltsurf} and \ref{ex:ltstab}. Moreover, the explicit knowledge of the excess constant allowed us in \cite{stability} to prove stability of relativistic matter in magnetic fields for critical nuclear charges and for the physical value of the fine structure constant.


In the second part of this paper we will focus on the operator $H_\Omega(\mathbf{A})=(\mathbf{D}-\mathbf{A})^2$ defined on an open set $\Omega$ of finite measure with Dirichlet boundary conditions. The analog of \eqref{eq:ltintro} that we will study is
\begin{equation}
 \label{eq:blyintro}
\tr\left(H_\Omega(\mathbf{A})-\lambda\right)^\gamma_- \leq K_{\gamma,d}\, |\Omega|\, \lambda^{\gamma+d/2} \,.
\end{equation}
Following \cite{FrLoWe}, in Theorem \ref{blymagnonsharp} we derive explicit bounds on the constants $K_{\gamma,d}$ and in Theorem \ref{blyhom} we determine the \emph{sharp} value of the constant for $d=2$ if $\mathbf{A}$ is restricted to generate a homogeneous magnetic field. In particular, this will imply that for $0\leq\gamma<1$ P\'olya's conjecture is violated in the presence of a magnetic field and that for tiling domains $\Omega$ the constant in the magnetic case is \emph{strictly larger} than that in the non-magnetic estimate. This shows that, at least in the abstract setting of Theorem \ref{diamag}, one cannot expect the magnetic estimate to have the same constant as the non-magnetic estimate. In Subsection \ref{sec:homneu} we demonstrate that if Dirichlet boundary conditions are replaced by Neumann boundary condition (and if $\mathbf{A}$ is again restricted to generate a homogeneous magnetic field), then a sharp inequality for the eigenvalues holds in the reverse sense.

The material in Section \ref{sec:abs} is new (extending \cite{Ro,hltsimple}) and we provide complete proofs, while most of the material in Section \ref{sec:blymag} has previously appeared in \cite{ErLoVo,FrLoWe,FrHa,magdomain} and we only sketch the arguments. We hope that this presentation shows some common aspects behind the different results.

\subsection*{Acknowledgments} 
Most of the results reviewed here were obtained in collaborations with A. Hansson, A. Laptev, E. Lieb, M. Loss, S. Molchanov, R. Seiringer, and T. Weidl, and it is a great pleasure to thank them for many interesting discussions. I would also like to thank G. Rozenblum and M. Solomyak for providing me with references and the organizers of the conference `Spectral and Scattering Theory for Quantum Magnetic Systems' in Luminy for their kind invitation.


\section{An abstract result}\label{sec:abs}

In this section we shall discuss the question raised in the introduction in a more general setting. We formulate the main result in Subsection \ref{sec:absmain} and prove it in Subsection \ref{sec:absproof}. Subsections \ref{sec:absdisc} and \ref{sec:absex} contain an improvement under more restrictive assumptions and examples, respectively.

\subsection{Assumptions and main result}\label{sec:absmain}

Let $X$ be a sigma-finite measure space and $H$ and $M$ two self-adjoint, non-negative operators in $L_2(X)$ with corresponding quadratic forms $h$ and $m$. Our crucial assumption is that for any $f\in L_2(X)$ and any $t>0$ one has
\begin{equation}\label{eq:domination}
|\exp(-tM) f(x)| \leq (\exp(-tH)|f|)(x)
\qquad \text{a.e.}\ x\in X \,,
\end{equation}
that is, the semigroup of $H$ is positivity preserving and dominates that of $M$.

\begin{remark}
In applications assumption \eqref{eq:domination} can be verified in terms of the corresponding quadratic forms. Indeed, the inequality $\exp(-tH)f\geq 0$ for all non-negative $f\in L_2(X)$ is equivalent to the following two conditions,
\begin{equation}\label{eq:bd1}
\text{for any}\ u\in\dom h \ \text{one has}\ \re u, \im u\in\dom h \ \text{and}\ h[\re u,\im u]\in\R
\end{equation}
and
\begin{equation}\label{eq:bd2}
\text{for any real-valued}\ u\in\dom h \ \text{one has}\ |u|\in\dom h \ \text{and}\ h[|u|] \leq h[u] \,.
\end{equation}
Moreover, \eqref{eq:domination} is equivalent to \eqref{eq:bd1}, \eqref{eq:bd2} and
\begin{equation}\label{eq:kato}
\begin{split}
\text{for any}\ u\in\dom m \ \text{and}\ v\in\dom h \ \text{with}\ 0\leq v\leq |u| \ \text{one has}\\
 |u|\in\dom h \,,
v \,\sgn u \in\dom m \ \text{and}\ h[v,|u|] \leq \re m[v\sgn u, u] \,.
\end{split}
\end{equation}
(Here we use the definition $\sgn u(x):= u(x)/|u(x)|$ if $u(x)\neq 0$ and $\sgn u(x):=0$ if $u(x)=0$. Moreover, $h[\cdot,\cdot]$ denotes the sesqui-linear form associated to the quadratic form $h[\cdot]$ which is anti-linear in the first and linear in the second argument.) These equivalences are essentially due to Beurling and Deny and to Hess, Schrader, Uhlenbrock and Simon; see \cite[Sec. 2]{Ou} for proofs and references. (\cite[Thm. 2.21]{Ou} requires \eqref{eq:kato} for arbitrary (not only non-negative) $v$, but the same proof yields the stated result.)
\end{remark}

In order to define the perturbed operator, let $Y$ be a sigma-finite measure space and $G$ an (unbounded) operator from $L_2(X)$ to $L_2(Y)$ such that $\dom G\subset \dom h$ and such that the quadratic form $\| G u\|_{L_2(Y)}^2$ is form-bounded with respect to $H$ with relative form-bound zero. Under these condition the quadratic form $h[u] - \lambda \| G u\|_{L_2(Y)}^2$, $u\in\dom h$, is closed for any $\lambda>0$ and generates a self-adjoint operator in $L_2(X)$ which we will denote by $H- \lambda G^* G$. (Strictly speaking, this is an abuse of notation, since we do not require $G$ to be closable and $G^*$ to be densely defined  -- indeed, it is not in our Example \ref{ex:ltsurf}.)

In addition, we assume that $G$ has the following reality and positivity properties.
\begin{equation}\label{eq:greal}
\text{for any}\ u\in\dom G \ \text{one has}\ \re u\in\dom G \ \text{and} \ (G\re u, G\im u)\in\R
\end{equation}
and
\begin{equation}\label{eq:gkato}
\begin{split}
\text{for any}\ u,v \in\dom G \ \text{with}\ 0\leq v\leq |u| \ \text{one has}\
|u|\in\dom G \,, v\,\sgn u \in\dom G \,,\\  |Gu| = G|u| \ \text{a.e.} \ \text{and} \ 
(Gv,G|u|) = \re (G(v \sgn u),Gu) \,.
\end{split}
\end{equation}
It follows from \eqref{eq:kato} and \eqref{eq:gkato} that $\| G u\|_{L_2(Y)}^2$ is also form-bounded with respect to $M$ with relative form-bound zero and hence $m[u] - \lambda \| G u\|_{L_2(Y)}^2$, $u\in\dom m$, generates a self-adjoint operator $M-\lambda G^*G$ in $L_2(X)$.

\bigskip

The following theorem states that a power-like bound on the number ($\gamma=0$) or moments ($\gamma>0$) of negative eigenvalues of $H-\lambda G^*G$ implies a similar bound for those of $M-\lambda G^*G$ with a larger, but explicit and $M$-independent constant. For $\gamma=0$ this observation is due to Rozenblum \cite{Ro}. The proof given below modifies and extends his arguments to cover the case $\gamma>0$. A slightly less general result for $\gamma>0$ has appeared in \cite{hltsimple}.

\begin{theorem}\label{diamag}
Under the above assumptions suppose that for some $\gamma,\alpha \geq 0$ and $C>0$ one has
\begin{equation}\label{eq:diamagnegass}
\tr\left(H-\lambda G^*G\right)_-^\gamma \leq C \lambda^\alpha
\qquad \text{for all}\ \lambda\geq 0 \,.
\end{equation}
Then one has
\begin{equation}\label{eq:diamag}
\tr\left(M-\lambda G^*G \right)_-^\gamma \leq C \, \left(\frac e\alpha\right)^\alpha \Gamma(\alpha+1) \, \lambda^\alpha
\qquad \text{for all}\ \lambda\geq 0 \,.
\end{equation}
\end{theorem}

In \eqref{eq:diamag} we use the convention that $(e/\alpha)^\alpha=1$ if $\alpha=0$.

We do not claim that the excess factor $\left(e/ \alpha\right)^\alpha \Gamma(\alpha+1)$ is sharp. We shall see in Subsection \ref{sec:hom} below, however, that in general the estimate $\tr\left(M-\lambda G^*G \right)_-^\gamma \leq C' \lambda^\alpha$ holds only with a constant $C'$ which is \emph{strictly larger} than the (sharp) $C$ in \eqref{eq:diamagnegass}.

Here is an important special case for which some of the steps in the proof of Theorem~\ref{diamag} are simpler than in the general case.

\begin{example}
Let $X=Y$ with the same measure and let $V$ be a non-positive measurable function on $X$ such that multiplication by $V$ is form-bounded with respect to $H$ with form-bound zero. Then the operator $Gu :=\sqrt{V_-} u$ satisfies all the assumptions of this subsection.
\end{example}

Specializing even further we have

\begin{example}[Lieb-Thirring inequalities]\label{diamaglt}
 In the situation of the previous example assume that $X=\R^d$ with Lebesgue measure, $d\geq 2$, $H:=-\Delta$ and $V\in L_{\gamma+d/2}(\R^d)$. Then, as recalled in the introduction, \eqref{eq:diamagnegass} holds for $\gamma>0$ if $d=2$ and for $\gamma\geq 0$ if $d\geq 3$ with constant $C=L_{\gamma,d} \int_{\R^d} V_-^\alpha \,dx$ and $\alpha:=\gamma+d/2$. If $\mathbf{A}\in L_{2,\loc}(\R^d)$ and $M:=(\mathbf{D}-\mathbf{A})^{2}$, then the diamagnetic inequality \eqref{eq:diamagintro} and hence \eqref{eq:domination} hold. Therefore all the assumptions of this subsection are satisfied. While in this setting Theorem~\ref{diamag} does not lead to any new inequalities or improvement for the constants in the magnetic case, we will see in Subsection \ref{sec:absex} several examples modelled after this one where we indeed obtain new inequalities from Theorem \ref{diamag}.
\end{example}

\begin{remark}
 Similarly as in \cite{Ro} there is a more general statement which can be proved in the same way as Theorem \ref{diamag}. Namely, if \eqref{eq:diamagnegass} is replaced by the assumption that
\begin{equation*}
\tr\left(H-\lambda G^*G\right)_-^\gamma \leq \phi(\lambda)
\qquad \text{for all}\ \lambda\geq 0 \,,
\end{equation*}
for some non-negative, non-decreasing function $\phi$ of subexponential growth, then for all $t>0$
\begin{equation*}
\tr\left(M-\lambda G^*G\right)_-^\gamma \leq \frac{t e^t}{\lambda} \hat\phi(t/\lambda)
\qquad \text{for all}\ \lambda\geq 0 \,,
\end{equation*}
where $\hat\phi$ is the Laplace transform of $\phi$. In particular, if $\phi$ is regular in the sense that $\hat\phi(t)\leq C_\phi t^{-1} \phi(t^{-1})$ for some $C_\phi$ and all $t>0$, then
\begin{equation*}
\tr\left(M-\lambda G^*G\right)_-^\gamma \leq e C_\phi \phi(\lambda)
\qquad \text{for all}\ \lambda\geq 0 \,.
\end{equation*}
\end{remark}


\subsection{Proof of Theorem \ref{diamag}}\label{sec:absproof}

We denote by $N(-\tau,A)$ the dimension of the spectral subspace corresponding to the interval $(-\infty,-\tau)$ of a lower semi-bounded self-adjoint operator $A$. That is, for $-\tau\leq \inf\text{ess-}\spec A$, $N(-\tau,A)$ is the number of eigenvalues (counting multiplicities) less than $-\tau$. The key step in the proof of Theorem \ref{diamag} is the observation that $N(-\tau, M-G^*G)$ can be bounded from above by a constant times the average of $N(-\tau,H-\lambda G^*G)$ over all coupling constants $\lambda>0$. The measure $t e^{-\lambda t}\,d\lambda$ with respect to which we average depends on a parameter $t>0$.

\begin{lemma}\label{average}
Under the assumptions of Subsection \ref{sec:absmain} for any $\tau\geq 0$ and $t>0$ one has
\begin{equation}\label{eq:diamagnegproof}
N(-\tau,M- G^*G) \leq t e^t \int_0^\infty N(-\tau,H-\lambda G^*G) e^{-\lambda t}\,d\lambda \,.
\end{equation}
In particular, if $f$ is a non-negative, non-increasing and absolutely continuous function on $\R$ with $f(0)=0$, then
\begin{equation}\label{eq:averagetrace}
\tr f(M-G^*G) \leq t e^t \int_0^\infty \tr f(H-\lambda G^*G) e^{-\lambda t}\,d\lambda \,.
\end{equation}
\end{lemma}

Assuming this lemma for the moment we can easily complete the

\begin{proof}[Proof of Theorem \ref{diamag}]
We assume that $\gamma>0$, the argument for $\gamma=0$ being similar. By Lemma \ref{average} with $G$ replaced by $\sqrt\lambda G$ (which satisfies the same assumptions as $G$) and $f(s)=s_-^\gamma$ one has for any $t>0$
\begin{align*}
\tr\left(M-\lambda G^*G\right)_-^\gamma
\leq t e^t \int_0^\infty \tr\left(H-\mu \lambda G^*G\right)_-^\gamma e^{-\mu t}\,d\mu \,.
\end{align*}
By assumption \eqref{eq:diamagnegass} the right hand side can be bounded from above by
\begin{equation*}
t e^t C  \lambda^\alpha \int_0^\infty \mu^\alpha e^{-\mu t}\,d\mu 
= \lambda^\alpha t^{-\alpha} e^t \Gamma(\alpha+1) C \,,
\end{equation*}
and the assertion follows by choosing $t=\alpha$.
\end{proof}

The following proof of Lemma \ref{average} relies on some ideas from \cite{Ro}.

\begin{proof}[Proof of Lemma \ref{average}]
Since \eqref{eq:domination} remains valid with $H+\tau$ and $M+\tau$ in place of $H$ and $M$ we need only consider $\tau=0$. Moreover, by a limiting argument (which is only necessary if originally $\tau=0$) we may assume that $H$ and $M$ are positive definite. We consider the subspaces $\mathfrak h_H := \overline{\ran G H^{-1/2}}$ and $\mathfrak h_M := \overline{\ran G M^{-1/2}}$ of $L_2(Y)$ and denote the corresponding orthogonal projections by $P_H$ and $P_M$. By our assumptions the operators $\tilde K_H :=P_H GH^{-1/2}$ and $\tilde K_M := P_M G M^{-1/2}$ acting from $L_2(X)$ to $\mathfrak h_H$ and $\mathfrak h_M$, respectively, are bounded. The Birman-Schwinger principle implies that
\begin{align}\label{eq:bs}
N(0, H-\lambda G^*G) & = n(\lambda^{-1}, \left(G H^{-1/2} \right) \left(G H^{-1/2} \right)^* )=n(\lambda^{-1}, \left(G H^{-1/2} \right)^* \left(G H^{-1/2} \right) ) \notag \\
& = n(\lambda^{-1}, \tilde K_H^* \tilde K_H) = n(\lambda^{-1}, \tilde K_H \tilde K_H^*)
\end{align}
and similarly for $M$. Here $n(\lambda^{-1},A)$ denotes the dimension of the spectral subspace corresponding to the interval $(\lambda^{-1},\infty)$ of a self-adjoint operator $A$. Since $\tilde K_H$ and $\tilde K_M$ have dense ranges, their adjoints $\tilde K_H^*$ and $\tilde K_M^*$ have trivial kernels, and hence the operators $\tilde K_H \tilde K_H^*$ and $\tilde K_M \tilde K_M^*$ have self-adjoint (unbounded) inverses $A_H$ and $A_M$ in $\mathfrak h_H$ and $\mathfrak h_M$, respectively.
In terms of these operators the Birman-Schwinger principle \eqref{eq:bs} can be rewritten as
\begin{equation}\label{eq:bs2}
N(0, H-\lambda G^*G) = N(\lambda, A_H) \,,
\qquad 
N(0, M-\lambda G^*G) = N(\lambda, A_M) \,.
\end{equation}
We define the operators $T_H(t):= P_H^* \exp(-tA_H)P_H$ and $T_M(t):= P_M^* \exp(-tA_M)P_M$ on $L_2(Y)$ and claim that for all $t>0$ and all $f\in L_2(Y)$ one has
\begin{equation}\label{eq:domproof}
\left|\left(T_M(t)f\right)(y)\right| \leq \left(T_H(t) |f|\right)(y)
\qquad \text{a.e.}\ y\in Y \,.
\end{equation}
Accepting this for the moment, we deduce using \cite[Thm. 2.13]{Si3} that
$$
\tr_{L_2(Y)} T_M(t) = \|T_M(t/2)\|_2^2 \leq \| T_H(t/2) \|_2^2= \tr_{L_2(Y)} T_H(t) \,,
$$
where $\|\cdot\|_2$ denotes the Hilbert-Schmidt norm.
This together with \eqref{eq:bs2} implies
\begin{align*}
N(0, M-G^*G) & = N(1, A_M) \leq e^t \tr_{\mathfrak h_M} \exp(-t A_M) = e^t  \tr_{L_2(Y)} T_M(t) \\
& \leq e^t  \tr_{L_2(Y)} T_H(t) = e^t \tr_{\mathfrak h_H} \exp(-t A_H) \\
& = t e^t \int_0^\infty N(\lambda, A_H) e^{-t\lambda} \,d\lambda
= t e^t \int_0^\infty N(0,H-\lambda G^*G) e^{-t\lambda} \,d\lambda \,,
\end{align*}
which is the first assertion. Estimate \eqref{eq:averagetrace} follows by writing
$$
\tr f(A) = -\int_0^\infty f'(-\tau) N(-\tau,A)\,d\tau
$$
for any self-adjoint operator $A$ from \eqref{eq:diamagnegproof} by Fubini's theorem.

It remains to prove \eqref{eq:domproof}. Since this fact is proved in \cite{Ro}, we only sketch the major steps in the argument. We begin by showing that for any $t>0$, $T_H(t)$ or, what is the same, for any $\tau>0$, $P_H^*(A_H+\tau)^{-1} P_H$ is positivity preserving. As in \cite{Ro} one easily verifies that
$$
P_H^* (A_H+\tau)^{-1} P_H = K_H (1+\tau K_H^*K_H)^{-1} K_H^*
$$
with $K_H:=GH^{-1/2}$. Writing $K_H=\text{s-}\lim_{\epsilon\to 0}G \exp(-\epsilon H) H^{-1/2}$ and noting that $G \exp(-\epsilon H)$ and its adjoint are positivity preserving (see \eqref{eq:gkato}), we are left with proving that
$$
H^{-1/2} (1+\tau K_H^*K_H)^{-1} H^{-1/2} = (H+\tau G^*G)^{-1}
$$
is positivity preserving. Using that $\exp(-tH)$ is positivity preserving and recalling \eqref{eq:greal} and \eqref{eq:gkato} we deduce this from the Beurling-Deny conditions; see, e.g., \cite[Thm.~2.7]{Ou}.

Finally, we prove that for any $t>0$, $T_H(t)$ dominates $T_M(t)$, or equivalently, that for any $\tau>0$, 
$P_H^*(A_H+\tau)^{-1} P_H$ dominates $P_M^*(A_M+\tau)^{-1} P_M$. Arguing as before (see also \cite{Ro}) this will follow from the fact that $(H+\tau G^* G)^{-1}$ dominates $(M+\tau G^* G)^{-1}$, which in turn can be deduced from the form version of Kato's inequality using \eqref{eq:kato} and \eqref{eq:gkato}; see, e.g., \cite[Thm.~2.21]{Ou}. This concludes the proof of \eqref{eq:domproof} and hence that of Lemma~\ref{average}.
\end{proof}


\subsection{An improvement in the case of discrete spectrum}\label{sec:absdisc}

In this subsection we show that in the special case where $X=Y$ (with the same measure) and $G=I$ the excess factor $\left(e/\alpha\right)^\alpha \Gamma(\alpha+1)$ in Theorem \ref{diamag} can be improved for $\gamma>0$. Note that in this case \eqref{eq:diamagnegass} (or \eqref{eq:diamagass} below) requires $H$ to have purely discrete spectrum.

\begin{theorem}\label{diamagdisc}
Under the assumptions of Subsection \ref{sec:absmain} suppose that for some $\alpha\geq\gamma\geq 0$ and $C>0$ one has
\begin{equation}\label{eq:diamagass}
\tr(H-\lambda)_-^\gamma \leq C \lambda^\alpha
\qquad \text{for all}\ \lambda\geq 0 \,.
\end{equation}
Then one has
\begin{equation}\label{eq:diamagdisc}
\tr(M-\lambda)_-^\gamma \leq C \, \left(\frac\gamma e\right)^\gamma \left(\frac e\alpha\right)^\alpha \frac{\Gamma(\alpha+1)}{\Gamma(\gamma+1)} \, \lambda^\alpha
\qquad \text{for all}\ \lambda\geq 0 \,.
\end{equation}
\end{theorem}

In \eqref{eq:diamagdisc} we use the convention that $(\gamma/e)^\gamma=1$ if $\gamma=0$ and similarly for $\alpha=0$. As in Theorem \ref{diamag} we do not claim that the excess factor $\left(\frac\gamma e\right)^\gamma \left(\frac e\alpha\right)^\alpha \frac{\Gamma(\alpha+1)}{\Gamma(\gamma+1)}$ in \eqref{eq:diamagdisc} is sharp, but we have examples (for $0\leq \gamma<1$, $\alpha=\gamma+1$), where it is larger than one; see Subsection \ref{sec:hom} below.

\begin{remark}
 As we shall see in Lemma \ref{goingdown}, \eqref{eq:diamagass} implies $N(\lambda)\leq C' \lambda^{\alpha-\gamma}$ for some constant $C'$. Conversely, the integration argument of Aizenman-Lieb \cite{AiLi} shows that $N(\lambda)\leq C' \lambda^{\alpha-\gamma}$ implies \eqref{eq:diamagass} for some $C$.
\end{remark}

\begin{proof}
According to \cite[Thm. 2.13]{Si3} the domination property \eqref{eq:domination} yields
\begin{equation}\label{eq:diamagproof}
\tr\exp(-tM) = \|\exp(-tM/2)\|_2^2\leq \|\exp(-tH/2)\|_2^2 = \tr\exp(-tH) \,,
\end{equation}
where $\|\cdot\|_2$ denotes the Hilbert-Schmidt norm. In order to estimate the right side from above we use the elementary formula $\Gamma(\gamma+1) e^{-\lambda} = \int_0^\infty (\lambda-\mu)_-^\gamma e^{-\mu}\,d\mu$, which gives
\begin{equation*}\label{eq:semigroup}
\tr\exp(-tH) = \frac{t^{\gamma+1}}{\Gamma(\gamma+1)} \int_0^\infty \tr(H-\mu)_-^\gamma e^{-t\mu}\,d\mu \,.
\end{equation*}
Hence by assumption \eqref{eq:diamagass} we have
\begin{align*}
\tr\exp(-tH)
\leq \frac{C t^{\gamma+1}}{\Gamma(\gamma+1)} \int_0^\infty \mu^\alpha e^{-t\mu}\,d\mu
= C \, \frac{\Gamma(\alpha+1)}{\Gamma(\gamma+1)}\, t^{-\alpha+\gamma} \,.
\end{align*}
In order to estimate the left side of \eqref{eq:diamagproof} from below we use that $\lambda_-^\gamma \leq (\gamma/e)^\gamma e^{-\lambda}$, which implies that
\begin{equation*}
\exp(-t M) \geq \left(\frac e \gamma \right)^{\gamma}  t^{\gamma} e^{-t\lambda} \tr(M-\lambda)_-^\gamma\,.
\end{equation*}
Combining these two estimates with \eqref{eq:diamagproof} we find
\begin{equation*}
\tr(M-\lambda)_-^\gamma 
\leq C \left(\frac \gamma e\right)^{\gamma} \frac{\Gamma(\alpha+1)}{\Gamma(\gamma+1)} e^{t\lambda} t^{-\alpha} \,.
\end{equation*}
We optimize the right side by choosing $t=\alpha/\lambda$ and obtain the assertion.
\end{proof}


\subsection{Examples}\label{sec:absex}

\subsubsection{An endpoint estimate in 2D}\label{ex:ltlog}

It is well-known that the Lieb-Thirring inequality \eqref{eq:ltintro} does \emph{not} hold for $\gamma=0$ if $d=2$. For $\mathbf{A}\equiv 0$ a replacement was recently found by Kova\v r\'{\i}k, Vugalter and Weidl \cite{KoVuWe}. It involves the quantity $\tr f(l^2(-\Delta+V))$ with
$$
f(s):=
\begin{cases}
1& \text{if}\ s\leq - e^{-1} \,, \\
|\ln |s||^{-1} & \text{if}\ - e^{-1}< s<0 \,, \\
0& \text{if}\ s\geq 0 \,.
\end{cases}
$$
In \cite{hltsimple} we used a version of Theorem \ref{diamag} to extend the estimate to the magnetic case.

\begin{theorem}
Let $d=2$ and $f$ as above. Then there exists a constant $L>0$ and for any $q>1$ a constant $L_q>0$ such that for all $l>0$, $V\in L_1(\R^2,\log_+(l/|x|)dx)\cap L_1(\R_+,rdr, L_q(\Sph))$ and $\mathbf{A}\in L_{2,\loc}(\R^2,\R^2)$ one has
$$
\tr f\left(l^2((\mathbf{D}-\mathbf{A})^2+V)\right)
\leq L \int_{|x|<l} \! V(x)_- \log\frac l{|x|} \,dx + L_q \int_0^\infty \!\!\left(\int_{\Sph} V(r\omega)_-^q \,d\omega \right)^{1/q} \!r\,dr \,.
$$
\end{theorem}

\begin{proof}
As explained in Example \ref{diamaglt} we are in the situation of Subsection \ref{sec:absmain}. By Lemma \ref{average} and the result of \cite{KoVuWe} one has for $V\leq 0$ and $t>0$
\begin{align*}
&\tr f\left(l^2((\mathbf{D}-\mathbf{A})^2+V)\right) \leq t e^t \int_0^\infty \tr f\left(l^2(-\Delta +\mu V)\right) e^{-\mu t}\,dt \\
& \leq t e^t \left(L' \int_{|x|<l} \! V(x)_- \log\frac l{|x|} \,dx + L_q' \int_0^\infty \!\!\left(\int_{\Sph} V(r\omega)_-^q \,d\omega \right)^{1/q} \!r\,dr \right) \int_0^\infty \mu e^{-\mu t} \,d\mu \\
& \leq t^{-1} e^t \left(L' \int_{|x|<l} \! V(x)_- \log\frac l{|x|} \,dx + L_q' \int_0^\infty \!\!\left(\int_{\Sph} V(r\omega)_-^q \,d\omega \right)^{1/q} \!r\,dr \right) \,.
\end{align*}
We obtain the assertion by choosing $t=1$.
\end{proof}

\subsubsection{Lieb-Thirring inequalities for surface potentials}\label{ex:ltsurf}

Our next example concerns Schr\"o\-dinger operators in $\R^{d+1}$ with potentials supported on the hyperplane $\R^d\times\{0\}$. Let $\mathbf{A}\in L_{2,\loc}(\R^{d+1},\R^{d+1})$ and let the operator $H(\mathbf{A},v)$ be defined through the closure of the quadratic form
\begin{equation}\label{eq:ltsurfform}
  \iint_{\R^{d+1}} |(D-\mathbf{A}) u|^2 \,dx\,dy + \int_{\R^d} v(x)|u(x,0)|^2\,dx \,,
  \quad u\in C_0^\infty(\R^{d+1})\, .
\end{equation}
Schr\"odinger operators with interactions supported on lower dimensional manifolds have been studied extensively and we refer to \cite{BrExKuSe} for motivations and references. The fact that the number of negative eigenvalues of the operator $H(\mathbf{A},v)$ satisfies a Cwikel-Lieb-Rozenblum inequality was found by Rozenblum \cite{Ro}. In \cite{ltsurf} we proved Lieb-Thirring inequalities for this operator in the case $\mathbf{A}\equiv 0$. Here we will use Theorem \ref{diamag} to extend these inequalities to arbitrary $\mathbf{A}$. We emphasize that in this application of Theorem \ref{diamag} one has $X\neq Y$.

\begin{theorem}
  \label{ltsurf}
  Let $\gamma>0$ if $d=1$ and $\gamma\geq 0$ if $d\geq 2$. Then there exists a constant $S_{\gamma,d}$ such that for any $v\in L_{2\gamma+d}(\R^d)$ and all $\mathbf{A}\in  L_{2,\loc}(\R^{d+1},\R^{d+1})$ one has
  \begin{equation}\label{eq:main1}
    \tr\left[ H(\mathbf{A},v) \right]_-^\gamma
    \leq S_{\gamma,d} \int_{\R^d} v(x)_-^{2\gamma+d} \,dx \,.
  \end{equation}
\end{theorem}

\begin{proof}
 We take $X:=\R^{d+1}$ with Lebesgue measure. As explained in Example \ref{diamaglt} the diamagnetic inequality \eqref{eq:domination} holds for $H:=-\Delta$ and $M:=(\mathbf{D}-\mathbf{A})^2$. To define the perturbation let $0\geq v\in L_{2\gamma+d}(\R^d)$ and define $Y:=\R^d$ with Lebesgue measure and $(Gu)(y):=\sqrt{v(y)_-} u(y,0)$ with $\dom G:= H^1(\R^{d+1})$. Note that $G$ is well-defined by the Sobolev trace theorem and satisfies the assumptions in Subsection \ref{sec:absmain}. Hence the assertion follows from Theorem \ref{diamag} and the result for $\mathbf{A}\equiv 0$ in \cite{ltsurf}.
\end{proof}

\begin{remark}
 Similarly as in \cite{ltsurf}, Theorem \ref{ltsurf} implies a theorem about operators on the halfspace $\R^{d+1}_+ =\{(x,y):\ x\in\R^d,\, y>0 \}$. Indeed, let $\mathbf A$ be given on $\R^{d+1}_+$ and let $\tilde H(\mathbf A,v)$ in $L_2(\R^{d+1}_+)$ be defined through the quadratic form \eqref{eq:ltsurfform} for $u\in C_0^\infty(\overline{\R^{d+1}_+})$ with the first integral restricted to $\R^{d+1}_+$. Extending $\mathbf A$ to $\R^{d+1}$ by setting $A_j(x,y)=A_j(x,-y)$ for $j=1,\ldots,d$ and $A_{d+1}(x,y)=-A_{d+1}(x,-y)$, we see that the operator $H(\mathbf A, 2v)$ leaves the subspaces of even and odd functions with respect to $y$ invariant, and that its parts on even and odd functions are unitarily equivalent to $\tilde H(\mathbf A,v)$ and to the Dirichlet Laplacian on $\R^{d+1}_+$, respectively. Hence Lieb-Thirring inequalities for $\tilde H(\mathbf A,v)$ follow immediately from those for $H(\mathbf A, 2v)$.
\end{remark}

\subsubsection{Subtracting a critical local singularity}\label{ex:ltstab}

Let $d=3$ and $\mathbf{A}\in L_{2,\loc}(\R^3)$. We claim that the quadratic form
\begin{equation}\label{eq:formstab}
\left\| |\mathbf{D}-\mathbf{A}|^{1/2} u\right\|^2 -\frac2\pi \left\| |x|^{-1/2} u \right\|^2
\end{equation}
is non-negative for $u\in C_0^\infty(\R^3)$. Here $|\mathbf{D}-\mathbf{A}| := \sqrt{(\mathbf{D}-\mathbf{A})^2}$ is defined via the spectral theorem. Indeed, if $\mathbf{A}\equiv 0$ this is Kato's inequality (see, e.g., \cite{He} for a proof). For general $\mathbf{A}$ we combine the diamagnetic inequality \eqref{eq:diamagintro} and the subordination formula
\begin{equation*}\label{eq:subord}
e^{-\lambda} = \frac{1}{\sqrt\pi} \int_0^\infty e^{-s-\lambda^2/(4s)} \frac{ds}{\sqrt s}
\end{equation*}
to obtain
\begin{equation}
 \label{eq:diamagrel}
|\exp(-t \left|\mathbf{D}-\mathbf{A}\right| ) f| \leq \exp(-t \sqrt{-\Delta} ) |f|
\quad \text{a.e.}
\end{equation}
This implies $\| |\mathbf{D}-\mathbf{A}|^{1/2} u\|^2 \geq \| (-\Delta)^{1/4} u\|^2$ for all $u\in C_0^\infty(\R^3)$ and hence the non-negativity of \eqref{eq:formstab}.

Now let $\Omega\subset\R^3$ be an open set. The form \eqref{eq:formstab} restricted to 
$$\{ u \in\dom \sqrt{|\mathbf{D}-\mathbf{A}|-(2/\pi)|x|^{-1} } :\ u\equiv 0 \ \text{on}\ \Omega^c \}$$ is non-negative and closed in $L_2(\Omega)$ (since limits of functions that are zero on $\Omega^c$ are zero on $\Omega^c$), and hence generates a non-negative operator $T_\Omega(\mathbf{A})$ in $L_2(\Omega)$. For this operator one has

\begin{theorem}
 \label{stability}
Let $\Omega\subset\R^3$ be an open set of finite measure and $\mathbf{A}\in L_{2,\loc}(\R^3)$. Then
\begin{equation}\label{eq:stability}
    \tr\left( T_{\Omega}(\mathbf{A}) -\lambda \right)_-
\leq 2.0152\, \lambda^4 \,|\Omega| 
\qquad\text{for all}\ \lambda>0\,.
  \end{equation}
\end{theorem}

This estimate is the key ingredient in the proof of stability of relativistic matter in magnetic fields in \cite{stability}. The constant $2/\pi$ multiplying the singularity $|x|^{-1}$ corresponds to the critical nuclear charge. The constant on the right side of \eqref{eq:stability} determines the allowed range $(0,\alpha_c)$ of the fine structure constant, and the value 2.0152 leads to $\alpha_c=1/133$ which is larger than the physical value $1/137\,$ !

The analog of \eqref{eq:stability} for relativistic Schr\"odinger operators $|\mathbf{D}-\mathbf{A}| -\frac2\pi |x|^{-1} +V$ was proved in \cite{ltpseudo}, extending previous work of \cite{EkFr}. Remarkably, these inequalities lead to semi-classical bounds even though the classical phase-space integral diverges due to the singularity $|x|^{-1}$. The inequalities have extensions to arbitrary dimensions and to arbitrary fractional powers of the Laplacian, see \cite{ltpseudo,hltsimple}.

\begin{proof}
 We apply Theorem \ref{diamagdisc} with $X=\Omega$, $H=T_\Omega(0)$ and $M=T_\Omega(\mathbf{A})$. In order to prove the diamagnetic inequality \eqref{eq:domination} we note that by Trotter's product formula \eqref{eq:diamagrel} remains valid if $|\mathbf{D}-\mathbf{A}|$ and $\sqrt{-\Delta}$ are replaced by $|\mathbf{D}-\mathbf{A}|+V$ and $\sqrt{-\Delta}+V$. Choosing $V=n\chi_{\Omega^c}$ and letting $n\to\infty$ the operators converge to $T_\Omega(\mathbf{A})$ and $T_\Omega(0)$ in strong resolvent sense, which yields \eqref{eq:domination}; see \cite{stability} for details. Inequality \eqref{eq:diamagass} for $\mathbf{A}\equiv 0$, $\gamma=1$ and $\alpha=4$ was shown in \cite{LiYa} with $C=(3/4\pi)\times4.4827 \, |\Omega|$. (In \cite{LiYa} it is assumed that $\Omega$ is a ball, but the same proof applies to any open set of finite measure.) Theorem \ref{diamagdisc} yields the assertion with constant $C'=6(e/4)^3 C$.
\end{proof}


\section{Semi-classical spectral estimates for magnetic Laplacians}\label{sec:blymag}

In this section we assume that $\Omega\subset\R^d$, $d\geq 2$, is an open set of finite measure and that $\mathbf{A}\in L_{2,\loc}(\Omega)$. We denote by $H_\Omega(\mathbf{A})$ the self-adjoint operator in $L_2(\Omega)$ corresponding to the closure of the quadratic form $\int_\Omega |(\mathbf{D}-\mathbf{A})u|^2 \,dx$
defined for $u\in C_0^\infty(\Omega)$. We are interested in estimates of the form
\begin{equation}\label{eq:blymag}
\tr\left(H_\Omega(\mathbf{A})-\lambda\right)_-^\gamma 
  \leq \rho_{\gamma,d} L_{\gamma,d}^{\cl} \lambda^{\gamma+d/2} |\Omega| \,,
  \qquad \lambda\geq 0 \,,
\end{equation}
with the semi-classical constant
\begin{equation}\label{eq:constsc}
L_{\gamma,d}^{\cl}
=\frac{\Gamma(\gamma+1)}{2^d\pi^{d/2}\Gamma(\gamma+\frac{d}{2}+1)}\,.
\end{equation}
Our goal will be to find optimal or close to optimal values for $\rho_{\gamma,d}$ and we begin by recalling some known facts concerning this problem.

\begin{enumerate}
\item Estimate \eqref{eq:blymag} holds for any $\gamma\geq 0$ with some finite constant $\rho_{\gamma,d}$ depending only on $d$ and $\gamma$. For $\mathbf{A}\equiv 0$ this was independently shown by Lieb \cite{Li2}, M\'etivier \cite{Me} and Rozenblum \cite{Ro1}. Lieb's proof works also for non-trivial $\mathbf{A}$.
 \item The constant $\rho_{\gamma,d}$ in \eqref{eq:blymag} cannot be less than one. This is a consequence of the asymptotics $\lambda^{-\gamma-d/2} \tr\left(H_\Omega(\mathbf{A})-\lambda\right)_-^\gamma \to L_{\gamma,d}^{\cl} |\Omega|$ as $\lambda\to\infty$. (Since $\lambda\to\infty$ is equivalent to $\hbar\to0$ for $\mathbf{A}\equiv 0$, this explains why $L_{\gamma,d}^\cl$ is called the \emph{semi-classical constant}.) We refer to the appendix for references and a short proof under our minimal assumptions on $\mathbf{A}$ and $\Omega$.
\item By an argument of Aizenman and Lieb \cite{AiLi} one can show that (the smallest possible) $\rho_{\gamma,d}$ is a non-increasing function of $\gamma$.
\item A celebrated result of Laptev and Weidl \cite{LaWe1} (see also \cite{BeLo}) implies that \eqref{eq:blymag} holds with $\rho_{\gamma,d}=1$ if $\gamma\geq 3/2$. (More precisely, they considered Schr\"odinger operators $(\mathbf{D}-\mathbf{A})^2+V$ in the whole space, but this implies \eqref{eq:blymag} by the variational principle; see the proof of Theorem \ref{blymagnonsharp} for a technical subtlety in this argument.)
\item In the case $\mathbf{A}\equiv 0$, Berezin \cite{Be1} and Li and Yau \cite{LY} have independently shown that \eqref{eq:blymag} holds with $\rho_{\gamma,d}=1$ if $\gamma\geq 1$.
\item In the case $\mathbf{A}\equiv 0$ and $\Omega$ tiling (that is, $\R^d$ can be decomposed, up to a set of measure zero, into a disjoint union of translated and rotated copies of $\Omega$), P\'olya \cite{Po} has proved \eqref{eq:blymag} with $\rho_{\gamma,d}=1$ for all $\gamma\geq 0$. That this is true without the tiling assumption is an open conjecture.
\end{enumerate}

We emphasize that the analogs of the Berezin-Li-Yau result and the P\'olya result for arbitrary magnetic fields are not known. In this section we shall review some recent progress concerning the constants $\rho_{\gamma,d}$ in \eqref{eq:blymag}. In particular, it was shown in \cite{ErLoVo} and \cite{FrLoWe} that, if $d=2$ and $\mathbf{A}$ is restricted to generate a homogeneous magnetic field, \eqref{eq:blymag} holds with $\rho_{\gamma,2}=1$ for $\gamma\geq 1$, but one needs $\rho_{\gamma,2}>1$ for $0\leq\gamma<1$. This means that P\'olya's conjecture is not true in the magnetic case; see Subsection \ref{sec:hom}.


\subsection{Arbitrary magnetic fields}

To begin our investigation of \eqref{eq:blymag} we show that an idea similar to that in Theorem \ref{diamagdisc} allows one to derive explicit values for the constants for $0\leq\gamma<3/2$ from those for $\gamma=3/2$.

\begin{theorem}\label{blymagnonsharp}
Let $\Omega\subset\R^d$, $d\geq 2$, be an open set of finite measure and $\mathbf{A}\in L_{2,\loc}(\Omega)$. Then for $0\leq \gamma<3/2$ one has
\begin{equation}\label{eq:blymagnonsharpd}
  \tr\left(H_\Omega(\mathbf{A})-\lambda\right)_-^\gamma 
  \leq \rho_{\gamma,d} L_{\gamma,d}^{\cl} \lambda^{\gamma+d/2} |\Omega|
\qquad\text{for all}\ \lambda>0
\end{equation}
with
\begin{equation*}
  \rho_{\gamma,d} =
  \frac{\Gamma(5/2)\, \Gamma(\gamma+d/2+1)}{\Gamma((5+d)/2)\, \Gamma(\gamma+1)}
  3^{-3/2} (3+d)^{(3+d)/2} (2\gamma)^\gamma (2\gamma+d)^{-\gamma-d/2} \,.
\end{equation*}
\end{theorem}

This has appeared in \cite{FrLoWe}. We do not claim that the values of the constants $\rho_{\gamma,d}$ are best possible. We note, however, that for $d=2$ one has
\begin{equation}
\label{eq:blymagnonsharpconst2}
\rho_{\gamma,2} = (5/3)^{3/2} (\gamma/(\gamma+1))^\gamma
\end{equation}
and, in particular, $\rho_{1,2}= (5/3)^{3/2}/2\approx 1.076$ and $\rho_{0,2}=(5/3)^{3/2}\approx 2.152$. It will follow from Theorem \ref{blyhom} below that for any $0\leq\gamma\leq 1$ the constant $\rho_{\gamma,2}$ is off at most by a factor of $(5/3)^{3/2}/2\approx 1.0758\,.$

Our proof is based on the following abstract lemma (see \cite{FrLoWe}) which allows one to go from larger values of $\gamma$ to smaller ones. It is somewhat similar in spirit to the estimate $\tr(H-\lambda)_-^\gamma\leq \left(\frac \gamma e\right)^{\gamma}  t^{-\gamma} e^{-t\lambda} \exp(-t H)$ used in the proof of Theorem \ref{diamagdisc}.

\begin{lemma}\label{goingdown}
 Let $H$ be a non-negative self-adjoint operator with discrete spectrum and assume that for some $\sigma>0$, $\kappa\geq 0$ and $C>0$ one has
\begin{equation}\label{eq:discest0}
\tr(H-\lambda)_-^\sigma \leq C \lambda^{\sigma+\kappa}
\qquad\text{for all}\ \lambda>0 \,.
\end{equation}
Then for any $0\leq\gamma<\sigma$ one has
\begin{equation}\label{eq:discest}
\tr(H-\lambda)_-^\gamma \leq C \ \frac{b(\gamma,\sigma)}{b(\gamma+\kappa,\sigma+\kappa)} \ \lambda^{\gamma+\kappa}
\qquad\text{for all}\ \lambda>0 \,,
\end{equation}
where $b(0,\sigma):= 1$ if $\sigma>\gamma=0$ and $b(\gamma,\sigma):= \sigma^{-\sigma}\gamma^\gamma (\sigma-\gamma)^{\sigma-\gamma}$ if $\sigma>\gamma >0$.
\end{lemma}

\begin{proof}[Proof of Lemma \ref{goingdown}]
We first note that for $\sigma>\gamma\geq 0$, $\mu>\lambda$ and $E\geq 0$ one has
\begin{equation}\label{eq:gd}
(E-\lambda)_-^\gamma \leq b(\gamma,\sigma) (\mu-\lambda)^{-\sigma+\gamma} (E-\mu)_-^\sigma \,.
\end{equation}
with $b(\gamma,\sigma)$ as given in the lemma. Indeed, this follows by maximizing $(\mu-E)^{-\sigma} (\lambda-E)^\gamma$ explicitly over $E\in (0,\lambda)$. Combining \eqref{eq:gd} and \eqref{eq:discest0} we infer that for any $\mu>\lambda$
$$
\tr(H-\lambda)_-^\gamma \leq b(\gamma,\sigma) (\mu-\lambda)^{-\sigma+\gamma} \tr(H-\mu)_-^\sigma
\leq C b(\gamma,\sigma) (\mu-\lambda)^{-\sigma+\gamma} \mu^{\sigma+\kappa} \,.
$$
Optimizing the right side by choosing $\mu=\lambda (\sigma+\kappa)/(\gamma+\kappa)$ we obtain the assertion.
\end{proof}

\begin{remark}
 A slight variation of this argument shows that \eqref{eq:discest0} with $\kappa<0$ implies that $H=0$.
\end{remark}

\begin{proof}[Proof of Theorem \ref{blymagnonsharp}]
First, assume that $\mathbf{A}\in L_{2,\loc}(\overline{\Omega})$. Then the extension of $\mathbf{A}$ by $0$ belongs to 
$L_{2,\loc}(\R^d)$ and the Laptev-Weidl result \cite{LaWe1} together with the variational principle yield \eqref{eq:discest0} with $H=H_\Omega(\mathbf{A})$, $\sigma=3/2$, $\kappa=d/2$ and $C= L_{\sigma,d}^\cl |\Omega|$. The assertion in this case follows from Lemma~\ref{goingdown} and the explicit expression \eqref{eq:constsc} of the semi-classical constants.

Now assume only that $\mathbf{A}\in L_{2,\loc}(\Omega)$ and choose $\mathbf{A}_n\in L_{2,\loc}(\overline{\Omega})$ such that $\mathbf{A}_n\to \mathbf{A}$ in $L_{2,\loc}(\Omega)$. Following closely the arguments in \cite{Ka} or \cite{Si79} (the analog of the proof of \cite[Thm. 4.1]{Si79} is even simpler since we are considering the minimal operators) one shows that $H_\Omega(\mathbf{A}_n)\to H_\Omega(\mathbf{A})$ in strong resolvent sense. Hence for any $\gamma\geq 0$ and any $\lambda>0$, $\left(H_\Omega(\mathbf{A}_n)-\lambda\right)_-^\gamma \to \left(H_\Omega(\mathbf{A})-\lambda\right)_-^\gamma$ strongly \cite[Thm. VIII.20]{ReSi1} and by Fatou's lemma for trace ideals \cite[Thm. 2.7]{Si3}
$\liminf_{n\to\infty}\tr\left(H_\Omega(\mathbf{A}_n)-\lambda\right)_-^\gamma \geq \tr\left(H_\Omega(\mathbf{A})-\lambda\right)_-^\gamma$. Thus the inequality for $H_\Omega(\mathbf{A})$ follows from that for $H_\Omega(\mathbf{A}_n)$.
\end{proof}


\subsection{Sharp local trace estimates}\label{sec:diamagloc}

In Subsections \ref{sec:hom} and \ref{sec:ab} below we present improved versions of inequalities \eqref{eq:blymagnonsharpd} for some particular choices of $\mathbf{A}$. These improvements rely on certain generalized diamagnetic inequalities which will be the subject of this subsection. To motivate these inequalities we first note that the standard diamagnetic inequality implies that for $\phi(\lambda)=\exp(-t\lambda)$, $t>0$, one has
\begin{equation}\label{eq:diamagloc}
\tr\left(\chi_\Omega \ \phi\left((\mathbf{D}-\mathbf{A})^2\right) \chi_\Omega\right) 
\leq \tr\left(\chi_\Omega \ \phi(-\Delta) \chi_\Omega\right)
\end{equation}
for any open $\Omega\subset\R^d$ of finite measure. (This follows by applying \cite[Thm. 2.13]{Si3} to the Hilbert-Schmidt operator $\phi\left((\mathbf{D}-\mathbf{A})^2\right)^{1/2} \chi_\Omega$.) By linearity, \eqref{eq:diamagloc} holds for any function $\phi$ of the form $\phi(\lambda)= \int_0^\infty e^{-t\lambda} d\mu(t)$ with $d\mu$ a non-negative measure. It is a natural question whether a similar inequality is valid for a more general class of non-negative functions $\phi$. That this is, indeed, the case for homogeneous magnetic fields is a beautiful observation of Erd\H os, Loss and Vougalter \cite{ErLoVo}.

\begin{proposition}\label{diamaglochom}
Let $d=2$ and $\mathbf{A}(x)=\frac B2(-x_2,x_1)^T$ for some $B>0$. Then \eqref{eq:diamagloc} holds for any non-negative convex function $\phi$ on $[0,\infty)$ with $\int_0^\infty \phi(\lambda)\,d\lambda<\infty$ and any open set $\Omega\subset\R^2$ of finite measure.
\end{proposition}

A similar statement holds in arbitrary dimension, but below we shall use it only for $d=2$. We include the short proof for the sake of completeness.

\begin{proof}
One has $\phi( (\mathbf{D}-\mathbf{A})^2 ) = \sum_{k=1}^\infty \phi((2k-1) B) P_k$ where $P_k$ is the projection onto the $k$-th Landau level, and therefore
$$
\tr\left(\chi_\Omega \,\phi\left((\mathbf{D}-\mathbf{A})^2\right) \chi_\Omega\right)
= \sum_{k=1}^\infty \phi((2k-1) B) \| P_k\chi_\Omega \|_2^2 \,.
$$
Recalling that $P_k(y,y)=B/2\pi$ for all $y\in\R^2$ one finds
$$
\| P_k\chi_\Omega \|_2^2 = \iint_{\R^2\times\R^2} |P_k(x,y)|^2 \chi_\Omega(y) \,dx\,dy
= \int_{\R^2} P_k(y,y) \chi_\Omega(y) \,dy = \frac B{2\pi} |\Omega| \,,
$$
and hence
$$
\tr\left(\chi_\Omega \,\phi\left((\mathbf{D}-\mathbf{A})^2\right) \chi_\Omega\right)
= \frac B{2\pi} |\Omega| \sum_{k=1}^\infty \phi((2k-1) B) \,.
$$
One the other hand, using the Fourier transform one easily finds that
$$
\tr\left(\chi_\Omega \,\phi\left(-\Delta\right) \chi_\Omega\right)
= \frac 1{(2\pi)^2} |\Omega| \int_{\R^2} \phi(|\xi|^2) \,d\xi 
= \frac 1{4\pi} |\Omega| \int_0^\infty \phi(\lambda) \,d\lambda \,.
$$
The assertion now follows from the mean value property of convex functions, i.e.,
$$
\phi((2k-1)B) \leq \frac{1}{2B} \int_{2(k-1)B}^{2kB} \phi(\lambda) \,d\lambda
$$
for all $k$.
\end{proof}

We note, in particular, that the function $\phi(\mu)=(\mu-\lambda)_-$ for fixed $\lambda>0$ satisfies the assumptions of Lemma \ref{diamaglochom}. As we will explain in the following subsection, Erd\H os, Loss and Vougalter \cite{ErLoVo} used this function in order to obtain sharp bounds on eigenvalue sums for Laplacians with a homogeneous magnetic field in a domain. They also point out that for this choice of $\phi$, \eqref{eq:diamagloc} might \emph{fail} if the magnetic field is not homogeneous. Their counterexample is perturbative. (The assumption of radial symmetry of $\mathbf{A}$ on p. 905 of \cite{ErLoVo} should be dropped \cite{Er}.) Here is a non-perturbative result from \cite{FrHa} involving the Aharonov-Bohm operator.

\begin{proposition}\label{diamaglocab}
Let $d=2$, $\mathbf{A}(x)=\alpha |x|^{-2} (-x_2,x_1)^T$ for some $\alpha\in\R\setminus\Z$ and $\gamma>-1$. Then for any open $\Omega\subset\R^d$ of finite measure and any $\lambda>0$ one has
\begin{equation}
 \label{eq:diamaglocab}
\tr\left(\chi_\Omega \, \left((\mathbf{D}-\mathbf{A})^2-\lambda\right)_-^\gamma \chi_\Omega\right) 
\leq R_\gamma(\alpha)\ \tr\left(\chi_\Omega \, \left(-\Delta-\lambda\right)_-^\gamma \chi_\Omega\right) \,,
\end{equation}
where the constant
\begin{equation}
 \label{eq:diamaglocabconst}
R_\gamma(\alpha) :=
(\gamma+1) \sup_{s\geq 0} \sum_{n\in\Z} \int_0^1 (1-\mu)^\gamma \, J^2_{|n-\alpha|}(\sqrt\mu s)\,d\mu 
\end{equation}
is sharp and satisfies $R_\gamma(\alpha)>1$. Here $J_{|n-\alpha|}$ are Bessel functions.
\end{proposition}

To be more precise, because of the non-integrable singularity of $|\mathbf{A}|^2$ at the origin the operator $(\mathbf{D}-\mathbf{A})^2$ has to be defined by closing the quadratic form $\int_{\R^2} |(\mathbf{D}-\mathbf{A})u|^2\,dx$ on  $C_0^\infty(\R^2 \setminus\{0\})$. Note that this vector potential corresponds to the magnetic field $\curl \mathbf{A}=2\pi\alpha\delta_0$. We assume that $\alpha\in\R\setminus\Z$ since for integer $\alpha$ the magnetic field can be gauged away, making \eqref{eq:diamaglocab} trivially true with constant unity.

The fact that the sharp constant satisfies $R_\gamma(\alpha)>1$ means that the generalized diamagnetic inequality \eqref{eq:diamagloc} is violated. This effect is rather minute, however, since numerically
\begin{equation}
 \label{eq:diamaglocabconstnum}
R_0(\alpha)\leq 1.054 \,,
\qquad
R_1(\alpha)\leq 1.034 \,,
\qquad
R_2(\alpha)\leq 1.011
\end{equation}
for all $\alpha\in\R$. A variation of the Aizenman-Lieb argument \cite{AiLi} shows that $R_\gamma(\alpha)$ is non-increasing with respect to $\gamma$ for fixed $\alpha$. We have numerical evidence, but no proof, that $R_\gamma(\alpha)$ is an increasing function of $\alpha\in [0,1/2]$ for fixed $\gamma$.

For the proof of Proposition \ref{diamaglocab} we refer to \cite{FrHa} and restrict ourselves here to explaining the strategy. We note that \eqref{eq:diamaglocab} is equivalent to the pointwise bound
$$
\left((\mathbf{D}-\mathbf{A})^2-\lambda\right)_-^\gamma(x,x) 
\leq R_\gamma(\alpha)\ \left(-\Delta-\lambda\right)_-^\gamma(x,x)
\qquad
\text{for all $x\in\R^2$}
$$
for the corresponding integral kernels. Diagonalizing $(\mathbf{D}-\mathbf{A})^2$ explicitly we find that the left side is given by
$$
\frac1{4\pi} \sum_{n\in\Z} \int_0^\infty (\mu-\lambda)_-^\gamma \, J^2_{|n-\alpha|}(\sqrt\mu |x|)\,d\mu \,,
$$
which shows immediately that $R_\gamma(\alpha)$ is the sharp constant in \eqref{eq:diamaglocab}. The main point is to prove the strict inequality $R_\gamma(\alpha)>1$. Intuitively, this is a consequence of the fact that the spectral density $(4\pi)^{-1} \sum_{n\in\Z} J^2_{|n-\alpha|}(\sqrt\mu |x|)$ oscillates around its limiting value $(4\pi)^{-1}$ as $|x|\to\infty$. In \cite{FrHa} we make this precise by deriving the asymptotics
$$
\sum_{n\in\Z} \int_0^1 (1-\mu)^\gamma \, J^2_{|n-\alpha|}(\sqrt\mu s)\,d\mu
= \frac1{\gamma+1} - \Gamma(\gamma+1) \frac{\sin\alpha\pi}{\pi} \frac{\sin(2s-\tfrac12\gamma\pi)}{s^{2+\gamma}} 
+\mathcal O(s^{-3-\gamma})
$$
as $s\to\infty$ which, of course, implies $R_\gamma(\alpha)>1$. We note that in the case $\alpha\in\frac12+\Z$ one has $\sum_{n\in\Z} J^2_{|n-1/2|}(t) = \frac2\pi \int_0^{2t} s^{-1} \sin s \,ds$, which simplifies the proof of the asymptotics considerably.


\subsection{Sharp estimates in the case of a homogeneous magnetic field}\label{sec:hom}

In the case of a homogeneous magnetic field one can determine the sharp constants in \eqref{eq:blymagnonsharpd}. For $\gamma\geq1$ this is (in a slightly weaker, Legendre-transformed form) due to Erd\H os, Loss and Vougalter \cite{ErLoVo}, while the inequality for $0\leq\gamma<1$ appeared in \cite{FrLoWe}. Note that both papers contain results for $d\geq 3$ as well.

\begin{theorem}\label{blyhom}
Let $\Omega\subset\R^2$ be an open set of finite measure and $\mathbf{A}(x)=\frac B2(-x_2,x_1)^T$ for some $B>0$. Then for any $\gamma\geq 0$ one has
\begin{equation}\label{eq:blyhom}
  \tr\left(H_\Omega(\mathbf{A})-\lambda\right)_-^\gamma 
  \leq \rho_{\gamma,2}^\hom L_{\gamma,2}^{\cl} \lambda^{\gamma+1} |\Omega|
\qquad\text{for all}\ \lambda>0
\end{equation}
with the constant
\begin{equation*}
  \rho_{\gamma,2}^\hom =
  \begin{cases} 2 & \text{if} \ \gamma=0\,, \\
  2 \left(\gamma/(\gamma+1) \right)^\gamma & \text{if} \ 0<\gamma<1 \,,\\
  1 & \text{if} \ \gamma\geq 1\,.
 \end{cases}
\end{equation*}
The constant $\rho_{\gamma,2}^\hom$ is sharp in the following sense.
\begin{enumerate}
 \item For any $0\leq\gamma<1$, any bounded, open set $\Omega\subset\R^2$ and any $\epsilon>0$ there exist $B>0$ and $\lambda>0$ such that
$$
\tr\left(H_\Omega(\mathbf{A})-\lambda\right)_-^\gamma \geq (1-\epsilon) \rho_{\gamma,2}^\hom L_{\gamma,d}^{\cl} \lambda^{\gamma+1} |\Omega| \,.
$$
\item For any $\gamma\geq1$, any open set $\Omega\subset\R^2$ of finite measure, any $B>0$ and any $\epsilon>0$ there exists a $\lambda>0$ such that
$$
\tr\left(H_\Omega(\mathbf{A})-\lambda\right)_-^\gamma \geq (1-\epsilon) \rho_{\gamma,2}^\hom L_{\gamma,d}^{\cl} \lambda^{\gamma+1} |\Omega| \,.
$$
\end{enumerate}
\end{theorem}

As we have recalled at the beginning of this section, the Berezin-Li-Yau inequality \cite{Be1, LY} states that if $\gamma\geq1$, then
\begin{equation}\label{eq:bly}
\tr\left(-\Delta_\Omega-\lambda\right)_-^\gamma \leq L_{\gamma,2}^{\cl} \, \lambda^{\gamma+1} \, |\Omega|
\qquad\text{for all}\ \lambda>0
\end{equation}
for the eigenvalues of the Dirichlet Laplacian $-\Delta_\Omega$. The first part of Theorem \ref{blyhom} says that the same is true if a homogeneous magnetic field is added. Moreover, recall that P\'olya \cite{Po} has shown \eqref{eq:bly} for \emph{all} $\gamma\geq 0$ if $\Omega$ is \emph{tiling} and has conjectured that it is true for arbitrary $\Omega$. The second part of our Theorem \ref{blyhom} says that \eqref{eq:bly} is \emph{not} true for $0\leq\gamma<1$ if a homogeneous magnetic field, not even if $\Omega$ is tiling. This provides examples which show that in the abstract Theorems \ref{diamag} and \ref{diamagdisc} one does not have the same constant in the magnetic and the non-magnetic estimates.

The sharpness of the constants in Theorem \ref{blyhom} has different origins for $\gamma\geq1$ and for $0\leq\gamma<1$. As the proof below will show, for $\gamma\geq1$ it is attained in the limit $\lambda\to\infty$ with $B$ arbitrary but fixed, while for $0\leq\gamma<1$ it is attained in the limit $\lambda\to\infty$ and $B\to\infty$ with $\lambda/B$ fixed at a certain, $\gamma$-dependent value. For a connection between our sharpness results and sharpness results for the harmonic oscillator \cite{dB,HeRo} we refer to \cite{FrLoWe}. We also emphasize that the sharpness result of Theorem \ref{blyhom} for $0\leq\gamma<1$ is stronger than the one stated in \cite{FrLoWe}.

\begin{proof}
We begin with the case $\gamma\geq 1$. The Berezin-Lieb inequality \cite{Be2,Li1} states that if $H$ is a self-adjoint operator and $P$ a projection, then for any convex $\phi$, $\tr P\phi(PHP)P \leq \tr P\phi(H)P$. (This is easily verified using Jensen's inequality.) We apply this with $H=H_{\R^2}(\mathbf{A})$, $P=\chi_\Omega$ and $\phi(\mu)=(\mu-\lambda)_-^\gamma$. Since $\chi_\Omega H_{\R^2}(\mathbf{A}) \chi_\Omega \leq H_\Omega(\mathbf{A})\oplus 0$ in $L_2(\R^2)=L_2(\Omega)\oplus L_2(\R^2\setminus\Omega)$, we obtain from Proposition \ref{diamaglochom}
\begin{align*}
\tr_{L_2(\Omega)}\left(H_\Omega(\mathbf{A})-\lambda\right)_-^\gamma 
& \leq \tr_{L_2(\R^2)}\left( \chi_\Omega \left(H_{\R^2}(\mathbf{A})-\lambda\right)_-^\gamma \chi_\Omega\right) \\
& \leq \tr_{L_2(\R^2)}\left( \chi_\Omega \left(-\Delta-\lambda\right)_-^\gamma \chi_\Omega\right)
= L_{\gamma,2}^{\cl} \lambda^{\gamma+1} |\Omega| \,.
\end{align*}
This is the claimed inequality for $\gamma\geq 1$. The inequality for $0\leq\gamma<1$ follows from that for $\gamma=1$ by Lemma \ref{goingdown}.

The sharpness of $\rho_{\gamma,2}^\hom$ for $\gamma\geq1$ follows immediately from the Weyl-type asymptotics recalled at the beginning of this section. The sharpness for $0\leq\gamma<1$ will be a consequence of the fact that for $0\leq\gamma<1$ and $B>0$ one has
\begin{equation}
 \label{eq:diamagsc}
\sup_{\lambda>0} \frac{\frac B{2\pi}\sum_{k\in\N} \left(B(2k-1)-\lambda\right)_-^\gamma}{L_{\gamma,2}^{\cl}
    \lambda^{\gamma+1}} = \rho_{\gamma,2}^\hom \,,
\end{equation}
as is easily verified, see also \cite{FrLoWe}. For $0<\gamma<1$ the supremum is attained for $\lambda=B(\gamma+1)$ and for $\gamma=0$ in the limit $\lambda\to B_+$. In order to prove the sharpness for $0<\gamma<1$ (the case $\gamma=0$ is similar and will be omitted) we recall that $\mu\mapsto\frac{b}{2\pi}\#\{k\in\N :\ b(2k-1)<\mu\}$ is the integrated density of states of the operator $(\mathbf{D}-b \mathbf{A}_1)^2$ in $L_2(\R^2)$, where $\mathbf{A}_1(x)=\frac 12(-x_2,x_1)^T$. This means that for any bounded, open set $\Omega$ and any $b$ and $\mu$ one has $|l\Omega|^{-1} \tr(H_{l\Omega}(b\mathbf{A}_1)-\mu)_-^\gamma\to \frac{b}{2\pi}\sum_{k\in\N} (b(2k-1)-\mu)_-^\gamma$ as $l\to\infty$. (Assuming only the boundedness of $\Omega$ this is shown in \cite{DoIwMi} -- indeed, in order to apply the results of that paper note that by dominated convergence $|\{x\in\Omega: \ \dist(x,\partial\Omega)<\epsilon\}|\to 0$ as $\epsilon\to0$ for any open set of finite measure; for related earlier results we refer to \cite{CV,Ta,Tr}.) Since $H_{l\Omega}(b\mathbf{A}_1)$ is unitarily equivalent to $l^{-2} H_\Omega(l^2 b\mathbf{A}_1)$ by scaling, we see that for any $\epsilon>0$ and $\mu>b$ there exists an $L=L(\epsilon,b,\mu,\Omega)$ such that for all $l\geq L$ one has
$$
\tr\left( H_{\Omega}(l^2 b\mathbf{A}_1)-l^2 \mu\right)_-^\gamma\geq (1-\epsilon) \, |\Omega| \, \frac{l^2 b}{2\pi}\sum_{k\in\N} \left(l^2 b(2k-1)-l^2 \mu\right)_-^\gamma \,.
$$
Hence by \eqref{eq:diamagsc} and the corresponding equality statement we see that with the choice $\mu=b(\gamma+1)$ one has
$$
\tr\left( H_{\Omega}(l^2 b\mathbf{A}_1)-l^2 \mu\right)_-^\gamma\geq (1-\epsilon) \, \rho_{\gamma,2}^\hom L_{\gamma,2}^\cl |\Omega| \, 
\left(l^2 \mu\right)^{\gamma+1} \,.
$$
This implies the assertion with $B=L^2 b$ and $\lambda=L^2\mu=B(\gamma+1)$.
\end{proof}

Here is an estimate with a right hand side depending on $B$.

\begin{theorem}\label{blyhommod}
Let $\Omega\subset\R^2$ be an open set of finite measure which is tiling and let $\mathbf{A}(x)=\frac B2(-x_2,x_1)^T$ for some $B>0$. Then for any $\gamma\geq 0$ and $\lambda>0$ one has
\begin{equation}\label{eq:blyhommod}
  \tr\left(H_\Omega(\mathbf{A})-\lambda\right)_-^\gamma 
  \leq |\Omega| \, \frac{B}{2\pi}\sum_{k\in\N} \left(B(2k-1)-\lambda\right)_-^\gamma \,.
\end{equation}
If $\gamma\geq 1$, then this is true without the assumption that $\Omega$ is tiling.
\end{theorem}

Inequality \eqref{eq:blyhommod} for $\gamma\geq1$ is \emph{stronger} than \eqref{eq:blyhom}. This follows from the mean value property of convex functions as in the proof of Proposition \ref{diamaglochom}. We do not know whether \eqref{eq:blyhommod} is true for $0\leq\gamma<1$ is $\Omega$ is not tiling.

\begin{proof}
 The inequality follows by P\'olya's original argument with the Weyl-type asymptotics replaced by the result about the integrated density of states mentioned in Theorem \ref{blyhom}; see \cite{FrLoWe} for details. The inequality for $\gamma\geq 1$ follows as in the proof of Theorem \ref{blyhom} from the Berezin-Lieb inequality, but this time we write out $\tr_{L_2(\R^2)}\left( \chi_\Omega \left(H_{\R^2}(\mathbf{A})-\lambda\right)_-^\gamma \chi_\Omega\right)$ explicitly; see the proof of Proposition \ref{diamaglochom}.
\end{proof}

Theorem \ref{blyhommod} suggests to look for good or optimal constants in the inequality
\begin{equation*}
 \tr\left((D_1+\tfrac B2x_2)^2 +(D_2-\tfrac B2x_1)^2 + V(x) \right)_-^\gamma
\leq \tilde\rho_{\gamma} \, \frac{B}{2\pi}\, \sum_{k\in\N} \int_{\R^2} \left(B(2k-1)+V(x) \right)_-^\gamma \,dx \,.
\end{equation*}
(For $\gamma=1$ this inequality follows from \cite[Thm. 5.1]{LiSoYn}, but this implies the inequality for all $\gamma\geq 1$ by the argument of \cite{AiLi}.) In particular, what is the sharp constant in this inequality if $V$ is restricted to be of the form $V(x)=\omega^2 |x|^2 -\lambda$? In this case the eigenvalues are known explicitly. The corresponding result for $B=0$ and $\gamma\geq 1$ is due to de la Bret\`eche \cite{dB}.


\subsection{Improved estimates in the case of an Aharonov-Bohm magnetic field}\label{sec:ab}

Similarly as in the previous subsection the generalized diamagnetic inequality from Proposition \ref{diamaglocab} leads to eigenvalues estimates for the Aharonov-Bohm operator in a domain.

\begin{theorem}\label{blyab}
Let $\Omega\subset\R^2$ be an open set of finite measure and $\mathbf{A}(x)=\frac{\alpha}{|x|^2} (-x_2,x_1)^T$ for some $\alpha\in\R\setminus\Z$. Then for any $1\leq \gamma<3/2$ one has
\begin{equation}\label{eq:blyab}
  \tr\left(H_\Omega(\mathbf{A})-\lambda\right)_-^\gamma 
  \leq R_\gamma(\alpha) L_{\gamma,2}^{\cl} \lambda^{\gamma+1} |\Omega|
\end{equation}
with the constant $R_\gamma(\alpha)$ from \eqref{eq:diamaglocabconst}.
\end{theorem}

Comparing \eqref{eq:diamaglocabconstnum} with \eqref{eq:blymagnonsharpconst2} we see that \eqref{eq:blyab} is a slight improvement over \eqref{eq:blymagnonsharpd} for $\gamma$ greater or equal, but close to $1$. Our proof below works also for $\gamma\geq 3/2$, but in this case the Laptev-Weidl result yields \eqref{eq:blyab} with $R_\gamma(\alpha)$ replaced by $1$. 

\begin{proof}
Similarly as in the first part of Theorem \ref{blyhom} the assertion follows from the Berezin-Lieb inequality and Proposition \ref{diamaglocab}.
\end{proof}


\subsection{Upper bounds on eigenvalues}\label{sec:magdomain}

So far in this section, we have been interested in lower bounds on eigenvalues of $H_\Omega(\mathbf{A})$, that is, in \emph{upper} bounds on $\tr\left(H_\Omega(\mathbf{A})-\lambda\right)_-^\gamma$. In this subsection we shall have a brief look at \emph{lower} bounds on this quantity. We assume that $\Omega\subset\R^d$ is open and that $H_\Omega(\mathbf{A})$ has discrete spectrum. For this assumption to hold it is sufficient that the Dirichlet Laplacian $H_\Omega(0)=-\Delta_\Omega$ has discrete spectrum (which is, in particular, the case if $\Omega$ has finite measure). Indeed, $-\Delta_\Omega$ having discrete spectrum is equivalent to $\exp(-t(-\Delta_\Omega))$ being compact. According to the diamagnetic inequality and a theorem of  Dodds, Fremlin and Pitt \cite{DoFr,Pi} this implies that $\exp(-tH_\Omega(\mathbf{A}))$ is compact, and hence that $H_\Omega(\mathbf{A})$ has discrete spectrum. Our lower bound has the following form.

\begin{theorem}\label{magdomain}
 Let $\Omega\subset\R^d$, $d\geq 2$, be an open set and let $\mathbf{A}\in L_{2,\loc}(\Omega)$ such that $H_\Omega(\mathbf{A})$ has discrete spectrum. Then for all $\gamma\geq 1$ and $\lambda\geq 0$ one has
  \begin{equation}\label{eq:magdomain}
   \tr\left(H_\Omega(\mathbf{A})-\lambda\right)_-^\gamma
    \geq \ell_{\gamma,d} \lambda_1^{-d/2}(\lambda-\lambda_1)_+^{\gamma+d/2}
  \end{equation}
where $\lambda_1=\inf\spec H_\Omega(\mathbf{A})$ and
\begin{equation}\label{eq:magdomainconst}
\ell_{\gamma,d}:=
\frac{\Gamma(\gamma+1)\ \Gamma(2+d/2)}{\Gamma(\gamma+1+d/2)} \
\frac{j_{(d-2)/2}^2 J_{d/2}^2(j_{(d-2)/2})}{d(d+2)} \,.
\end{equation}
Here $J_\nu$ denotes the Bessel function of order $\nu$ and $j_\nu$ its first positive zero.
\end{theorem}

Note that the right side of \eqref{eq:magdomain} has the same growth $\lambda^{\gamma+d/2}$ as the right side of \eqref{eq:blymagnonsharpd}. Moreover, $\lambda_1^{-d/2}$ replaces the quantity $|\Omega|$ in \eqref{eq:blymagnonsharpd}. We remark that \eqref{eq:magdomain} together with Yang's inequality yields universal (i.e., independent of $\Omega$) upper bounds on eigenvalue ratios $\lambda_k/\lambda_1$ of $H_\Omega(\mathbf{A})$ which have the optimal growth $k^{2/d}$; see \cite{magdomain}.

In the case $\gamma=1$, this bound was proved in \cite{magdomain} (even with an additional non-negative electric potential $V$), extending a result by \cite{Hr} for the case $\mathbf{A}\equiv 0$. Previously, Laptev \cite{La} had shown a similar inequality for $\mathbf{A}\equiv 0$ with $\lambda_1^{-d/2}$ on the right side replaced by a constant times $\|\omega\|_\infty^{-2}$ (where $-\Delta_\Omega\omega=\lambda_1\omega$ and $\|\omega\|=1$). The proof in \cite{La} can easily be generalized to non-trivial $\mathbf{A}$. Inequality \eqref{eq:magdomain} for $\gamma=1$ is then derived from an estimate of $\|\omega\|_\infty$ in terms of $\lambda_1^{d/4}$. This estimate in the case $\mathbf{A}\equiv 0$ is due to Chiti \cite{Ch} and can be extended to arbitrary $\mathbf{A}$ using Kato's diamagnetic inequality; see \cite{magdomain} for details.

In order to extend the inequality to $\gamma>1$ one can follow the strategy of \cite{AiLi}. Denoting by $B$ the beta function and using the inequality for $\gamma=1$ we obtain
\begin{align*}
 \tr\left(H_\Omega(\mathbf{A})-\lambda\right)_-^\gamma
& = \frac 1{B(2,\gamma-1)} \int_0^\infty \tr\left(H_\Omega(\mathbf{A})-\lambda+t\right)_- t^{\gamma-2} \,dt \\
& \geq \frac 1{B(2,\gamma-1)} \ell_{1,d} \lambda_1^{-d/2} \int_0^\infty (\lambda-t-\lambda_1)_+^{1+d/2} t^{\gamma-2} \,dt \\
& = \frac{B(2+d/2,\gamma-1)}{B(2,\gamma-1)} \ell_{1,d} \lambda_1^{-d/2} (\lambda-\lambda_1)_+^{\gamma+d/2} \,,
\end{align*}
which is the assertion.


\subsection{Sharp estimates in the case of a homogeneous magnetic field. The Neumann case}\label{sec:homneu}

In Subsection \ref{sec:hom} we have proved sharp lower bounds on the eigenvalues of the Dirichlet Laplacian with a homogeneous magnetic field. In this final subsection we shall prove sharp \emph{upper} bounds on the eigenvalues of the \emph{Neumann} Laplacian with a homogeneous magnetic field.

In the non-magnetic case it is known that bound \eqref{eq:bly} for the eigenvalues of the Dirichlet Laplacian has a (reverse) analogue for the Neumann Laplacian $-\Delta_\Omega^N$ (assuming that this operator has discrete spectrum). Indeed, P\'olya \cite{Po} (see also \cite{Ke}) has shown that if $\Omega\subset\R^2$ is \emph{tiling} and $\gamma\geq 0$, then
\begin{equation}\label{eq:blyneu}
\tr\left(-\Delta_\Omega^N-\lambda\right)_-^\gamma \geq L_{\gamma,2}^{\cl} \, \lambda^{\gamma+1} \, |\Omega|
\qquad\text{for all}\ \lambda>0 \,.
\end{equation}
Kr\"oger \cite{Kr} has shown that this bound is valid for arbitrary $\Omega$ provided $\gamma\geq 1$. Our goal is to extend these inequalities to the case of a homogeneous magnetic field.

Let $\Omega\subset\R^2$ be an open set and $\mathbf{A}(x)=\frac B2(-x_2,x_1)^T$ for some $B>0$. We denote by $H_\Omega^N(\mathbf{A})$ the self-adjoint operator in $L_2(\Omega)$ corresponding to the quadratic form $\int_\Omega |(\mathbf{D}-\mathbf{A})u|^2 \,dx$ defined for $u\in L_2(\Omega)\cap H^1_\loc(\Omega)$ such that $(D-A)u\in L_2(\Omega)$.

We shall assume that $H_\Omega^N(A)$ has discrete spectrum. A sufficient condition for this is that the Neumann Laplacian $-\Delta_\Omega^N$ as discrete spectrum. (This sufficiency follows by the same arguments as in Subsection \ref{sec:magdomain} from the diamagnetic inequality for $\exp(-tH_\Omega^N(A))$ in \cite{HuLeMuWa}.)

\begin{theorem}
 \label{homneu}
Let $\Omega\subset\R^2$ be an open set of finite measure which is bounded and tiling, and let $\mathbf{A}(x)=\frac B2(-x_2,x_1)^T$ for some $B>0$. If the spectrum of $H_\Omega^N(\mathbf A)$ is discrete, then for any $\gamma\geq 0$ and $\lambda>0$ one has
\begin{equation}
 \label{eq:homneu}
\tr\left(H_\Omega^N(\mathbf A)-\lambda\right)_-^\gamma 
\geq |\Omega| \, \frac{B}{2\pi}\sum_{k\in\N} \left(B(2k-1)-\lambda\right)_-^\gamma \,.
\end{equation}
If $\gamma\geq 1$, then this is true without the assumption that $\Omega$ is bounded and tiling.
\end{theorem}

\begin{remark}
 Theorem \ref{homneu} being an analogue of Theorem \ref{blyhommod}, we emphasize that the analogue of Theorem \ref{blyhom} is \emph{not} valid. Indeed, if $\Omega$ is, say, a Lipschitz domain and $\gamma\geq 0$, then there exists \emph{no} positive constant $c$ such that for all $\lambda>0$ and all $B>0$ one has
 \begin{equation}\label{eq:homneunot}
 \tr\left(H_\Omega^N(\mathbf A)-\lambda\right)_-^\gamma \geq c \lambda^{\gamma+1} \,.
 \end{equation}
 To see this, we recall that $\mu\mapsto\frac{b}{2\pi}\#\{k\in\N :\ b(2k-1)<\mu\}$ is the integrated density of states of the operator $(\mathbf{D}-b \mathbf{A}_1)^2$ in $L_2(\R^2)$, where $\mathbf{A}_1(x)=\frac 12(-x_2,x_1)^T$, and that the integrated density of states is independent of the choice of boundary conditions \cite{DoIwMi}. Using the same scaling argument as in the proof of Theorem \ref{blyhom} we see, in particular, that as $l\to\infty$
 $$
l^{-2(\gamma+1)} \tr\left( H_{\Omega}^N(l^2 b\mathbf{A}_1)-l^2 \mu\right)_-^\gamma \to 0
\qquad
\text{if}\ \mu< b \,.
$$
This contradicts a bound \eqref{eq:homneunot} with $c>0$. In passing we note that the asymptotics of $\tr\left( H_{\Omega}^N(l^2 b\mathbf{A}_1)-l^2 \mu\right)_-^\gamma$ as $l\to\infty$ are studied in \cite{magweyl}.
\end{remark}

\begin{proof}
 If $\Omega$ is tiling, the inequality follows as in \cite{FrLoWe} by Kellner's argument \cite{Ke} using the integrated density of states of the Landau Hamiltonian in the plane \cite{DoIwMi}.

 Now let $\Omega$ be arbitrary. Our proof of \eqref{eq:homneu} uses some ideas of \cite{Kr,La} (see also \cite{Be2,Li1}). We denote the eigenvalues of $H_\Omega^N(\mathbf A)$ by $\mu_j$ and the corresponding eigenfunctions by $u_j$. Moreover, let $P_k$ be the projection in $L_2(\R^2)$ onto the $k$-th Landau level. Then for any function $\phi$
$$
\tr\phi\left(H_\Omega^N(\mathbf A)\right) = \sum_{j,k} \phi(\mu_j) \|P_k u_j\|^2
= \sum_k \int_{\R^2} \int_0^\infty \phi(\mu) \,d(U_{k,z}, E(\mu) U_{k,z}) \,dz \,,
$$
where $dE(\mu)$ denotes the spectral measure of $H_\Omega^N(\mathbf A)$ and $U_{z,k}$ denotes the restriction of $P_k(\cdot,z)$ to $\Omega$. The measure $d(U_{k,z}, E(\mu) U_{k,z}) \,dz$ on $[0,\infty)\times\R^2$ satisfies
\begin{align*}
\int_{\R^2} \int_0^\infty \,d(U_{k,z}, E(\mu) U_{k,z}) \,dz
& = \int_{\R^2} \|U_{k,z}\|_{L_2(\Omega)}^2 \,dz
= \int_\Omega \left(\int_{\R^2} P_k(x,z) P_k(z,x) \,dz \right) \,dx \\
& = \int_\Omega P_k(x,x) \,dx = \frac B{2\pi} |\Omega| \,.
\end{align*}
Here we used the $P_k^2=P_k$ and that $P_k(x,x)=\frac B{2\pi}$. Hence, if $\phi$ is convex, then Jensen's inequality implies
\begin{align*}
\int_{\R^2} \int_0^\infty \phi(\mu) \,d(U_{k,z}, E(\mu) U_{k,z}) \,dz
& \geq \tfrac B{2\pi} |\Omega| \phi\left( \tfrac {2\pi}B |\Omega|^{-1} \int_{\R^2} \int_0^\infty \mu \,d(U_{k,z}, E(\mu) U_{k,z}) \,dz \right) \\
& = \tfrac B{2\pi} |\Omega| \phi\left( \tfrac {2\pi}B |\Omega|^{-1} \int_{\R^2} \int_\Omega |(\mathbf D-\mathbf A) U_{k,z} |^2 \,dx \,dz \right) \,.
\end{align*}
Hence the assertion will follow (with $\phi(\mu)=(\mu -\lambda)_-^\gamma$) if we show that for any $k$ and any $x\in\Omega$
\begin{equation}\label{eq:homneuproof}
\int_{\R^2} |(\mathbf D_x-\mathbf A(x)) U_{k,z}(x) |^2 \,dz = B(2k-1) \frac B{2\pi} \,.
\end{equation}
In order to prove this equality we denote $\mathbf Q_x := \mathbf D_x-\mathbf A(x)$. Since $P_k^2=P_k$, the left side of \eqref{eq:homneuproof} equals $\mathbf Q_x \overline{\mathbf{ Q}_y} P(x,y)|_{x=y}$. Since $\mathbf Q_x^2 P(x,y) = B(2k-1) P(x,y)$, and hence
$$\mathbf Q_x^2 P(x,y)|_{x=y} = B(2k-1)\frac B{2\pi}
\qquad \text{and} \qquad 
\overline{\mathbf Q_y}^2 P(x,y)|_{x=y} = B(2k-1)\frac B{2\pi}\,,
$$
it suffices to prove that
\begin{equation}\label{eq:kernelder}
\left(\mathbf Q_x^2+ \overline{\mathbf Q_y}^2 - 2 \mathbf Q_x \overline{\mathbf{ Q}_y} \right) P(x,y)|_{x=y} = 0 \,.
\end{equation}
Now we expand the $\mathbf Q_x$ and $\mathbf Q_y$ and write $\mathbf Q_x^2+ \overline{\mathbf Q_y}^2 - 2 \mathbf Q_x \overline{\mathbf{ Q}_y}$ as a sum of three terms, containing only derivatives of order zero, one and two, respectively. The zeroth order term is easily seen to vanish if $x=y$. The first order term is given by $- 2 \left(\mathbf A(x)-\mathbf A(y)\right)\cdot \left(\mathbf{D}_x + \mathbf{D}_y \right)$ and hence also vanishes if $x=y$. Thus \eqref{eq:kernelder} is equivalent to
$$
\left(\mathbf D_x^2+ \mathbf D_y^2 + 2 \mathbf D_x \mathbf{ D}_y \right) P(x,y)|_{x=y}
=0 \,.
$$
The latter equality follows by differentiating the identity $P_k(x,x)=\frac B{2\pi}$ twice with respect to $x$. This concludes the proof of \eqref{eq:homneuproof}.
\end{proof}


\appendix

\section{Weyl-type asymptotics}

The asymptotics for the eigenvalues of the Dirichlet Laplacian on a regular domain are a classical result of Weyl \cite{We}, which was extended by Birman and Solomyak \cite{BiSo} to any bounded open set and finally by Rozenblum \cite{Ro1} to any open set of finite measure (see also \cite{BiSobook}). In the magnetic case these asymptotics are a folk theorem, although it seems difficult to find precise assumptions on $\mathbf{A}$ and $\Omega$ in the literature; for results in the smooth case on the whole space see, e.g., \cite{CoScSe}. Here is a short proof, obtained in collaboration with R. Seiringer and based on ideas from \cite{Be2,Li1}, which requires only minimal assumptions on $\mathbf{A}$ and $\Omega$.

\begin{theorem}
 \label{weyl}
Let $\Omega\subset\R^d$ be an open set of finite measure and $\mathbf{A}\in L_{2,\loc}(\Omega)$. Then
$$
\lim_{\lambda\to\infty} \lambda^{-d/2} N(\lambda,H_\Omega(\mathbf{A})) = L_{0,d}^\cl \, |\Omega| \,.
$$
\end{theorem}

\begin{proof}
 By the Tauberian theorem (see, e.g., \cite[Thm. 10.3]{Sifunc}) it suffices to prove
\begin{equation}
 \label{eq:weylheat}
\lim_{t\to0} t^{d/2} \tr\exp(-tH_\Omega(\mathbf{A})) = (4\pi)^{-d/2} \, |\Omega| \,.
\end{equation}
It is well-known (see, e.g., \cite[Sec. 2.1]{Da}) that the semi-group of the Dirichlet Laplacian $H_\Omega(0)$ is trace class with $\tr\exp(-tH_\Omega(0)) \leq (4\pi t)^{-d/2} \, |\Omega|$. Using the diamagnetic inequality (which is valid in our setting, e.g., by the approximation argument given in the proof of Theorem \ref{blymagnonsharp}) and \eqref{eq:diamagproof} we infer that $\exp(-tH_\Omega(\mathbf{A}))$ is trace class with $\tr\exp(-tH_\Omega(\mathbf{A}))\leq\tr\exp(-tH_\Omega(0))$. This immediately gives the upper bound in \eqref{eq:weylheat}.

For the lower bound let $\Omega_{\delta}:= \{x\in\Omega: \dist(x,\partial\Omega)>\delta \}$. Since $|\Omega_{\delta}|\to|\Omega|$ as $\delta\to0$ by dominated convergence, it is enough to show that for any $\delta>0$
\begin{equation}
 \label{eq:weylheatlower}
\liminf_{t\to0} t^{d/2} \tr\exp(-tH_\Omega(\mathbf{A})) \geq (4\pi)^{-d/2} \, |\Omega_{\delta}| \,.
\end{equation}
For fixed $\delta>0$, let $g\in C_0^\infty(\R^d)$ be a real-valued function with support in the ball $\{|x|\leq\delta/2\}$ and $\|g\|=1$ and define the coherent states $F_{p,q}(x):= e^{ipx}g(x-q)$ for $p\in\R^d$ and $q\in\Omega_\delta$; see \cite[Chp. 12]{LiLo}. Plancherel's theorem implies the operator inequality in $L_2(\Omega)$
$$
0 \leq \iint_{\R^d\times\Omega_\delta} (\cdot,F_{p,q}) F_{p,q} \,\frac{dp\,dq}{(2\pi)^d} \leq 1
$$
and hence
$$
\tr \exp(-tH_\Omega(\mathbf{A})) \geq \iint_{\R^d\times\Omega_\delta} \left(F_{p,q}, \exp(-tH_\Omega(\mathbf{A})) F_{p,q} \right) \,\frac{dp\,dq}{(2\pi)^d} \,.
$$
The functions $F_{p,q}$ belong to the form domain of $H_\Omega(\mathbf{A})$ and
\begin{align*}
\int_\Omega |(\mathbf{D}-\mathbf{A})F_{p,q}|^2 \,dx = |p|^2 -2 p\,\cdot \langle \mathbf{A} \rangle(q) + \langle |\mathbf{A}|^2 \rangle(q)
+ \|\nabla g\|^2 \,,
\end{align*}
where $\langle \mathbf{A} \rangle(q) := \int_\Omega \mathbf{A}(x) g^2(x-q) \,dx$ and similarly for $\langle |\mathbf{A}|^2 \rangle(q)$ (which are finite for $\mathbf{A}\in L_{2,\loc}(\Omega)$). By Jensen's inequality for the spectral measure of $H_\Omega(\mathbf{A})$ we have
$\left(F_{p,q}, \exp(-tH_\Omega(\mathbf{A})) F_{p,q} \right)\geq \exp(-t \| H_\Omega(\mathbf{A})^{1/2} F_{p,q} \|^2 )$, and therefore
\begin{align*}
 \tr \exp(-tH_\Omega(\mathbf{A})) & \geq e^{-t \|\nabla g\|^2} \int_{\Omega_\delta} e^{-t \langle |\mathbf{A}|^2 \rangle(q)}
 \left( \int_{\R^d} e^{-t\left(|p|^2 -2 p\,\cdot \langle \mathbf{A} \rangle(q) \right)} \,\frac{dp}{(2\pi)^d} \right) \,dq \\
& = (4\pi t)^{-d/2} e^{-t \|\nabla g\|^2} \int_{\Omega_\delta} e^{-t \left(\langle |\mathbf{A}|^2 \rangle(q) - |\langle \mathbf{A} \rangle(q)|^2\right)} \,dq \,.
\end{align*}
Estimating $e^{t|\langle \mathbf{A} \rangle(q)|^2}\geq 1$ and using Jensen's inequality we obtain
$$
\tr \exp(-tH_\Omega(\mathbf{A})) 
\geq (4\pi t)^{-d/2} e^{-t \|\nabla g\|^2} |\Omega_\delta| \exp\left(-t |\Omega_\delta|^{-1} \int_{\Omega_\delta}  \langle |\mathbf{A}|^2 \rangle(q) \,dq \right) \,.
$$
This implies \eqref{eq:weylheatlower} provided $\int_{\Omega_\delta}  \langle |\mathbf{A}|^2 \rangle(q) \,dq <\infty$. To see that this integral is finite, we note that the set $\overline{\Omega_{\delta/2}}$ is bounded (for otherwise there were a sequence $(x_n)\subset \overline{\Omega_{\delta/2}}$ with $|x_n-x_{n-1}|\geq 1$, and then for $\delta<1$, $\Omega$ contained infinitely many disjoint balls $\{ y: |y-x_n|<\delta/2\}$, contradicting $|\Omega|<\infty$). Therefore $\mathbf{A}\in L_2(\overline{\Omega_{\delta/2}})$ and
$$
\int_{\Omega_\delta}  \langle |\mathbf{A}|^2 \rangle(q) \,dq \leq \|g\|_\infty^2 |\Omega_\delta| \int_{\overline{\Omega_{\delta/2}}} |\mathbf{A}(x)|^2 \,dx < \infty \,,
$$
completing the proof of the theorem.
\end{proof}


\bibliographystyle{amsalpha}

\begin{thebibliography}{HuLeM\"uWa}

\bibitem[AiLi]{AiLi} M. Aizenman, E. Lieb, \textit{On semiclassical
  bounds for eigenvalues of Schr\"odinger operators}. Phys. Lett.~A
  \textbf{66} (1978), no. 6, 427--429.

\bibitem[AvHeSi]{AvHeSi} J. Avron, I. Herbst, B. Simon, \textit{Schr\"odinger operators with magnetic fields. I. General interactions}. Duke Math. J. \textbf{45} (1978), no. 4, 847--883.

\bibitem[BeLo]{BeLo} R. Benguria, M. Loss, \textit{A simple proof of a theorem of Laptev and Weidl}. Math. Res. Lett.  \textbf{7}  (2000),  no. 2--3, 195--203.

\bibitem[Be1]{Be1} F. A.~Berezin, \textit{Covariant and contravariant
  symbols of operators} [Russian]. Math. USSR Izv. {\bf 6} (1972),
  1117--1151.
\bibitem[Be2]{Be2} F.~A.~Berezin, \textit{Convex functions of operators}
  [Russian]. Mat. Sb. \textbf{88} (1972), 268--276. 

\bibitem[BiSo1]{BiSo} M.S. ~Birman, M.Z. ~Solomjak, \textit{The principal term of the spectral asymptotics for ``non-smooth'' elliptic problems}. Functional Anal. Appl. \textbf{4} (1970), 265--275.

\bibitem[BiSo2]{BiSobook} M. S. Birman, M. Z. Solomjak, \textit{Quantitative analysis in Sobolev imbedding theorems and applications to spectral theory}. Amer. Math. Soc. Transl. Ser. 2, \textbf{114}. Amer. Math. Soc., Providence, RI, 1980.

\bibitem[BrExKu\v Se]{BrExKuSe}
J. F. Brasche, P. Exner, Yu. A. Kuperin, P. \v Seba, \textit{Schr\"odinger operators with singular interactions}.
J. Math. Anal. Appl. \textbf{184} (1994), no. 1, 112--139.

\bibitem[Ch]{Ch} G. Chiti, \textit{An isoperimetric inequality for the eigenfunctions of linear second order elliptic operators}. Boll. Un. Mat. Ital. (6) \textbf{1}-A, 1982, 145--151.

\bibitem[CV]{CV} Y.~Colin de Verdi\`ere, \textit{L'asymptotique de Weyl pour les bouteilles magn\'etiques} [French]. Comm. Math. Phys. \textbf{105} (1986), 327--335.

\bibitem[CoScSe]{CoScSe} J. M. Combes, R. Schrader, R. Seiler, \textit{Classical bounds and limits for energy distributions of Hamilton operators in electromagnetic fields}. Ann Physics \textbf{111} (1978), no. 1, 1--18. 

\bibitem[Cw]{Cw} M. Cwikel, \textit{Weak type estimates for singular values and the number of bound states of Schr\"odinger operators}.  Ann. Math. (2) \textbf{106} (1977), no. 1, 93--100.

\bibitem[Da]{Da} E. B. Davies, \textit{Heat kernels and spectral theory}. Cambridge Tracts in Mathematics \textbf{92}, Cambridge University Press, Cambridge, 1990.

\bibitem[dB]{dB} R. de la Bret\`eche, \textit{Preuve de la conjecture de Lieb-Thirring dans le cas des potentiels quadratiques strictement convexes} [French]. Ann. Inst. H. Poincar\'e Phys. Th\'eor. \textbf{70} (1999),  no. 4, 369--380.

\bibitem[DoFr]{DoFr} P. G. Dodds, D. H. Fremlin, \textit{Compact operators in Banach lattices}. Israel J. Math. \textbf{34} (1979), no. 4, 287--320.

\bibitem[DoIwMi]{DoIwMi} S. Doi, A. Iwatsuka, T. Mine, \textit{The uniqueness of the integrated density of states for the Schr\"odinger operators with magnetic fields}. Math. Z. \textbf{237} (2001), no. 2, 335--371.

\bibitem[EkFr]{EkFr} T. Ekholm, R. L. Frank, \textit{On Lieb-Thirring inequalities for Schr\"odinger operators with virtual level}. Comm. Math. Phys. \textbf{264} (2006), no. 3, 725--740.

\bibitem[Er]{Er} L.~Erd\H os, private communication.
\bibitem[ErLoVo]{ErLoVo} L.~Erd\H os, M.~Loss, V.~Vougalter,
  \textit{Diamagnetic behavior of sums of Dirichlet eigenvalues}. Ann. Inst. Fourier {\bf 50} (2000), 891--907.

\bibitem[Fr1]{magweyl} R. L. Frank, \textit{On the asymptotic number of edge states for magnetic Schr\"odinger operators}. Proc. London Math. Soc. (3) \textbf{95} (2007), no. 1, 1--19.

\bibitem[Fr2]{hltsimple} R. L. Frank, \textit{A simple proof of Hardy-Lieb-Thirring inequalities}. Comm. Math. Phys., to appear.
\bibitem[FrHa]{FrHa} R.~L.~Frank, A.~M.~Hansson, \textit{Eigenvalue estimates for the Aharonov-Bohm operator in a domain}. In: Methods of Spectral Analysis in Mathematical Physics, J. Janas, et al. (eds.), 35--50, Oper. Theory Adv. Appl. \textbf{186}, Birkh\"auser, Basel, 2008.
\bibitem[FrLa]{ltsurf} R. L. Frank, A. Laptev, \textit{Spectral inequalities for Schr\"odinger operators with surface potentials}. In: Spectral theory of differential operators, T. Suslina and D. Yafaev (eds.), 91--102, Amer. Math. Soc. Transl. Ser. 2, \textbf{225}, Amer. Math. Soc., Providence, RI, 2008.
\bibitem[FrLaMo]{magdomain} R. L. Frank, A. Laptev, S. Molchanov, \textit{Eigenvalue estimates for magnetic Schr\"odinger operators in domains}. Proc. Amer. Math. Soc. \textbf{136} (2008), no. 12, 4245--4255.

\bibitem[FrLiSe1]{ltpseudo} R. L. Frank, E. H. Lieb, R. Seiringer, \textit{Hardy-Lieb-Thirring inequalities for fractional Schr\"odinger operators}. J. Amer. Math. Soc. \textbf{21} (2008), no. 4, 925--950.

\bibitem[FrLiSe2]{stability} R. L. Frank, E. H. Lieb, R. Seiringer, \textit{Stability of relativistic matter with magnetic fields for nuclear charges up to the critical value}. Comm. Math. Phys. \textbf{275} (2007), no.~2, 479--489.
\bibitem[FrLoWe]{FrLoWe} R. L. Frank, M. Loss, T. Weidl, \textit{P\'olya's conjecture in the presence of a constant magnetic field}. J. Eur. Math. Soc., to appear.
\bibitem[HeRo]{HeRo} B. Helffer, D. Robert, \textit{Riesz means of bounded states and semi-classical limit connected with a Lieb-Thirring conjecture. II}. Ann. Inst. H. Poincar\'e Phys. Th\'eor. \textbf{53} (1990),  no. 2, 139--147.

\bibitem[He]{He} I. W. Herbst, \textit{Spectral theory of the operator $(p\sp{2}+m\sp{2})\sp{1/2}-Ze\sp{2}/r$}.    Comm. Math. Phys. \textbf{53} (1977), no. 3, 285--294.

\bibitem[Hr]{Hr} L. Hermi, \textit{Two new Weyl-type bounds for the Dirichlet Laplacian}. Trans. Amer. Math. Soc. \textbf{360} (2008), 1539--1558.

\bibitem[Hn]{Hu}
    D.~Hundertmark, \textit{Some bound state problems in quantum mechanics}. In: Spectral theory and 
    mathematical physics: a Festschrift in honor of Barry Simon's 60th birthday, 463--496, Proc. Sympos. Pure Math. \textbf{76}, Part 1, Amer. Math. Soc., Providence, RI, 2007.

\bibitem[HuLeM\"uWa]{HuLeMuWa} T. Hupfer, H. Leschke, P. M\"uller, S. Warzel, \textit{The absolute continuity of the integrated density of states for magnetic Schr\"odinger operators with certain unbounded random potentials}. Comm. Math. Phys. \textbf{221} (2001), no. 2, 229--254.

\bibitem[Ka]{Ka}
T. Kato, \textit{Remarks on Schr\"odinger operators with vector potentials}. Integral Eq. Operator Theory \textbf{1} (1978), 103--113.

\bibitem[Ke]{Ke}
R. Kellner, \textit{On a theorem of P\'olya}. Amer. Math. Monthly \textbf{73} (1966), 856--858. 

\bibitem[KoVuWe]{KoVuWe} H. Kova\v r\'{\i}k, S. Vugalter, T. Weidl, \textit{Spectral estimates for two-dimensional Schr\"odinger operators with application to quantum layers}. Comm. Math. Phys. \textbf{275} (2007), no. 3, 827--838.

\bibitem[Kr]{Kr} P. Kr\"oger, \textit{Upper bounds for the Neumann eigenvalues on a bounded domain in Euclidean space}.  J. Funct. Anal.  \textbf{106} (1992),  no. 2, 353--357.

\bibitem[La]{La} A. Laptev, \textit{Dirichlet and Neumann eigenvalue
    problems on domains in Euclidean
    spaces}. J. Funct. Anal. \textbf{151} (1997), no. 2, 531--545.

\bibitem[LaWe1]{LaWe1} 
  A. Laptev, T. Weidl, \textit{Sharp Lieb-Thirring inequalities in
  high dimensions}. Acta Math. \textbf{184} (2000), no.~1, 87--111.
\bibitem[LaWe2]{LaWe2} 
  A. Laptev, T. Weidl, \textit{Recent results on Lieb-Thirring
  inequalities}. Journ\'ees ``\'Equations aux D\'eriv\'ees Partielles"
  (La Chapelle sur Erdre, 2000), Exp. No. XX, Univ. Nantes, Nantes,
  2000.
\bibitem[LY]{LY} P.~Li, S.-T.~Yau, \textit{On the Schr\"odinger equation and
    the eigenvalue problem}. Comm. Math. Phys. {\bf 88} (1983),
    309--318.
\bibitem[Li1]{Li1} E.~H.~Lieb, \textit{The classical limit of quantum spin
    systems}. Comm. Math. Phys. \textbf{31} (1973), 327--340.
\bibitem[Li2]{Li2} E.~H.~Lieb, \textit{The number of bound states of one-body Schr\"odinger operators and the Weyl
  problem}. Proc. Sym. Pure Math. \textbf{36} (1980), 241--252. Announced in: \textit{Bounds on the eigenvalues of the Laplace and Schr\"odinger operators}.  Bull. Amer. Math. Soc. \textbf{82} (1976), no. 5, 751--753.

\bibitem[Li3]{Li3} E. H. Lieb, \textit{Flux phase of the half-filled band}. Phys. Rev. Lett. \textbf{73} (1994), 2158--2161.

\bibitem[LiLo]{LiLo} E. H. Lieb, M. Loss, \textit{Analysis. Second edition}. Graduate Studies in Mathematics \textbf{14}, Amer. Math. Soc., Providence, RI, 2001.

\bibitem[LiSoYn]{LiSoYn} E. H. Lieb, J. P. Solovej, J. Yngvason, \textit{Ground states of large quantum dots in magnetic fields}. Phys. Rev. B \textbf{51} (1995), no. 16, 10646--10665.

\bibitem[LiTh]{LiTh} E.~H.~Lieb, W. Thirring, \textit{Inequalities for the
  moments of the eigenvalues of the Schr\"odinger Hamiltonian and
  their relation to Sobolev inequalities}. In: Studies in Mathematical
  Physics, E. H. Lieb, et al. (eds.), 269--303. Princeton Univ. Press, Princeton, NJ, 1976.

\bibitem[LiYa]{LiYa} E. H. Lieb, H.-T. Yau, \textit{The stability and instability of relativistc matter}. Comm. Math. Phys. \textbf{118} (1988), 177--213.

\bibitem[M\'e]{Me}
  G.~M\'etivier, \textit{Valeurs propres de probl\`emes aux limites
  elliptiques irr\'eguliers}. Bull. Soc. Math. France,
  Mem. \textbf{51-52} (1977), 125--229.

\bibitem[Ou]{Ou} E. M. Ouhabaz, \textit{Analysis of heat equations on domains}. Princeton Univ. Press, Princeton, 2005.

\bibitem[Pi]{Pi} L. D. Pitt, \textit{A compactnesss condition for linear operators of function spaces}. J. Operator Theory \textbf{1} (1979), no. 1, 49--54.

\bibitem[P\'o]{Po} G.~P\'olya, \textit{On the eigenvalues of vibrating
  membranes}. Proc. London Math. Soc. {\bf 11} (1961), 419--433.
  
\bibitem[ReSi]{ReSi1} M.~Reed, B.~Simon, \textit{Methods of Modern
  Mathematical Physics}, vol. 1. Second edition. Academic Press, 1980.


\bibitem[Ro1]{Ro1} G.~V.~Rozenblum, \textit{On the eigenvalues of the first
    boundary value problem in unbounded domains}. Math. USSR-Sb. \textbf{18} (1972), 235--248.

\bibitem[Ro2]{Ro2} G.~V.~Rozenblum, \textit{Distribution of the discrete spectrum of singular differential operators}.
Soviet Math. (Iz. VUZ) \textbf{20} (1976), no. 1, 63--71. Announced in: Soviet Math. Dokl. \textbf{13} (1972), 245--249.

\bibitem[Ro3]{Ro} G.~V.~Rozenblum, \textit{Domination of semigroups and estimates for eigenvalues}. St. Petersburg Math. J. \textbf{12} (2001), no. 5, 831--845.

  \bibitem[Si1]{Si79} B. Simon, \textit{Maximal and minimal Schr\"odinger
     forms}. J. Operator Theory \textbf{1} (1979), no. 1, 37--47.
  
  \bibitem[Si2]{Si3} B. Simon, \textit{Trace ideals and their applications. Second edition}. Mathematical Surveys and Monographs {\bf 120}, Amer. Math. Soc., Providence, RI, 2005.

\bibitem[Si3]{Sifunc} B. Simon, \textit{Functional integration and quantum physics. Second edition}. AMS Chelsea Publishing, Providence, RI, 2005.

\bibitem[Ta]{Ta} H. Tamura, \textit{Asymptotic distribution of eigenvalues for Schr\"odinger operators with magnetic fields}.  Nagoya Math. J. \textbf{105}  (1987), 49--69.

\bibitem[Tr]{Tr} F. Truc, \textit{Semi-classical asymptotics for magnetic bottles}. Asymptot. Anal. \textbf{15} (1997), no. 3-4, 385--395.

\bibitem[We]{We} H.~Weyl, \textit{Das asymptotische Verteilungsgesetz der Eigenwerte linearer partieller  Differentialgleichungen}. Math. Ann. \textbf{71} (1911), 441--469.


\end{thebibliography}

\end{document}